\documentclass[letterpaper,9pt]{scrartcl}
\usepackage[body={5in,8.1in},
  top=1.2in, left=1.2in]{geometry}
\pdfoutput=1

\usepackage{amsmath}
\usepackage{amssymb}
\usepackage{accents}
\usepackage{mathtools}
\usepackage{graphicx}
\usepackage[export]{adjustbox}
\usepackage{array}
\usepackage{xcolor}
\usepackage{paralist}
\usepackage{caption}
\usepackage{subcaption}
\usepackage{tikz}
\usepackage{pgfplots}
\usetikzlibrary{plotmarks}
\usepackage{wrapfig}
\usepackage{placeins}
\usepackage{siunitx}
\usepackage{marvosym}
\usepackage{todonotes}
\usepackage[english]{babel}
\usepackage{wasysym}
\usepackage{multirow}
\usepackage{tabu}

\usepackage{enumitem}

\newtheorem{theorem}{Theorem}[section]

\newtheorem{corollary}[theorem]{Corollary}
\newtheorem{proposition}[theorem]{Proposition}
\newtheorem{definition}[theorem]{Definition}

\newenvironment{@abssec}[1]{%
     \if@twocolumn
       \section*{#1}%
     \else
       \vspace{.05in}\footnotesize
       \parindent .2in
         {\upshape\bfseries #1. }\ignorespaces 
     \fi}
     {\if@twocolumn\else\par\vspace{.1in}\fi}

\newcommand\AMSname{AMS subject classifications}
\newcommand\keywordsname{Key words}
\newenvironment{keywords}{\begin{@abssec}{\keywordsname}}{\end{@abssec}}
\newenvironment{AMS}{\begin{@abssec}{\AMSname}}{\end{@abssec}}

\usetikzlibrary{external}
\usepackage{chngcntr}
\counterwithin{figure}{section}
\numberwithin{equation}{section}
\numberwithin{table}{section}

\definecolor{Xred}{RGB}{228,26,28}
\definecolor{Xblue}{RGB}{14,102,174}
\definecolor{Xgreen}{RGB}{255,127,0}
\definecolor{Xorange}{RGB}{77,175,74}
\definecolor{Xviolet}{RGB}{152,78,163}

\pgfplotscreateplotcyclelist{btPlotCycle}{%
solid, color=Xred, mark size=1, every mark/.append style={solid, fill=Xred!50!white}, mark=square*\\%
solid, color=Xblue, mark=x\\%
solid, color=Xgreen, mark size=1.2, mark=diamond*\\%
solid, color=Xorange, mark size=1, mark=*\\%
}

\usepackage[
  algo2e,
  commentsnumbered,
  linesnumbered,
  ruled
]{algorithm2e}
\SetKwProg{Fn}{Function}{\string:}{end}
\SetKwProg{Find}{find}{\string:}{}
\SetKw{st}{such that}

\newtheorem{remark}[theorem]{Remark}

\newcommand{\encloseimage}[1]{
\setlength{\fboxrule}{0pt}%
\fbox{\raisebox{-.45\height}{
\setlength{\fboxrule}{1pt}%
#1
}}}

\newcommand{\wt}{\widetilde}
\newcommand{\dx}{\ \mathrm{d}x}
\newcommand{\R}{\ \mathbb{R}}

\newcommand{\subs}{V}

\newcommand{\subssd}{V}

\newcommand{\rsubssd}{\wt V}

\newcommand{\subsbd}{U}

\newcommand{\subsod}{O}

\newcommand{\subsmap}{P}

\newcommand{\puconstant}{c_\mathrm{pu}}
\newcommand{\puconstantV}{c_{\mathrm{pu},\tilde V}}
\newcommand{\ovlpconstant}{c_\mathrm{ovlp}}
\newcommand{\pufuncconstant}{{c_{\mathrm{pu}}^\prime}}
\newcommand{\diam}{\operatorname{diam}}

\newcommand{\extension}[1]{E(#1)}
\newcommand{\training}[1]{T(#1)}
\newcommand{\coupling}[1]{C(#1)}
\newcommand{\rcoupling}[1]{\widetilde C(#1)}

\newcommand{\extend}{\mathrm{Extend}}
\newcommand{\neigh}{\mathcal{N}}

\newcommand{\domains}{I}

\newcommand{\norm}[1]{{\left\|{#1}\right\|}}
\newcommand{\dualpair}[2]{{\left\langle #1 , #2 \right\rangle }}

\newcommand{\parspace}{\mathcal{P}}
\newcommand{\defaultparameter}{\overline{\mu}}

\newcommand{\basis}{{\mathcal{B}}}
\newcommand{\rbasis}{{\widetilde{\mathcal{B}}}}

\newcommand{\jhat}{\hat{\jmath}}

\newcommand{\argmax}{\operatornamewithlimits{arg\,max}}

\newcommand{\supp}{\operatorname{supp}}
\newcommand{\spanset}{\operatorname{span}}

\title{ArbiLoMod, a Simulation Technique
Designed for Arbitrary Local Modifications}

\author{Andreas Buhr \footnotemark[2]\ \footnotemark[3]\ \footnotemark[4]
\and
Christian Engwer \footnotemark[3]
\and
Mario Ohlberger \footnotemark[3]
\and
Stephan Rave \footnotemark[3]\ \footnotemark[5]
}

\begin{document}
\maketitle
\renewcommand{\thefootnote}{\fnsymbol{footnote}}

\footnotetext[2]{Corresponding author. \Letter \ andreas@andreasbuhr.de}
\footnotetext[3]{Institute for Computational and Applied Mathematics,
University of M\"unster, Einsteinstra\ss e 62, 48149 M\"unster, Germany
}
\footnotetext[4]{supported by CST Computer Simulation Technology AG}
\footnotetext[5]{supported by the German Federal Ministry of Education and Research (BMBF) under contract number 05M13PMA}
\begin{abstract}
Engineers manually optimizing a structure using Finite Element based simulation software often employ an iterative approach where in each iteration they change the structure slightly and resimulate.
Standard Finite Element based simulation software is usually not well suited for this workflow, as it restarts in each iteration, even for tiny changes.
In settings with complex local microstructure, where a fine mesh is required to capture the geometric detail,
localized model reduction can improve this workflow.
To this end, we introduce ArbiLoMod, a method which allows fast recomputation after arbitrary local modifications.
It employs a domain decomposition and a localized form of the Reduced Basis Method for model order reduction.
It assumes that the reduced basis on many of the unchanged domains can be reused after a localized change.
The reduced model is adapted when necessary, steered by a localized error indicator.
The global error introduced by the model order reduction is controlled by a robust and efficient localized a posteriori error estimator, certifying the quality of the result.
We demonstrate ArbiLoMod for a coercive, parameterized example with changing structure.
\end{abstract}


\begin{keywords}
model order reduction, reduced basis method, domain decomposition, a posteriori error estimation
\end{keywords}

\begin{AMS}65N55, 65N30\end{AMS}

\pagestyle{myheadings}
\thispagestyle{plain}
\markboth{A. BUHR AND C. ENGWER AND M. OHLBERGER AND S. RAVE}{ARBILOMOD}

\section{Introduction}
Finite Element based simulation is a standard tool
in many CAD/CAE assisted workflows in engineering.
Depending on the complexity of the design simulated,
on the underlying partial differential equation,
and on the desired fidelity of the approximation of the solution,
performing a simulation may take hours, days or even weeks.

ArbiLoMod aims at the acceleration of 
a very specific class
of problems, namely the 
repetitive simulation of
parameterized problems
with fine microstructure
without scale separation:
%
Problems which exhibit a microstructure
on a scale much smaller than the domain require a very fine
mesh to resolve the geometry and thus take very long to compute.
They often have much more degrees of freedom than necessary for 
the description of their physical behavior,
thus model order reduction can succeed.
For problems with scale separation, there are established methods
to reduce the model such as
MsFEM \cite{Hou97amultiscale}, HMM \cite{E02theheterogeneous}, VMM \cite{HUGHES19983}, or LOD \cite{Maalqvist2014}.
However, when there is no scale separation,
homogenization is not possible. 
ArbiLoMod is based on a localized RB approach which does not require scale separation.

In addition, engineers often want
to change the problem definition and resimulate several times
in an iterative process.
Their goal is not to come from a
problem definition to an approximation of the solution,
but to provide approximations to the solutions
of a sequence of problems, where each problem was created
by modifying the previous problem.
We assume that changes might be arbitrary, i.e. non parametric, but are of local nature, so that the majority
of the problem setting (e.g.\ large parts of the geometry) is unchanged between two simulation runs.
ArbiLoMod exploits the similarities between subsequent problems
by reusing local approximation spaces in regions of the computation domain unaffected by the change.

Furthermore, for many design applications in each iteration, the structure under 
consideration has to be simulated for a multitude of model parameters
to analyze its behavior. ArbiLoMod includes online offline
decomposition techniques from RB methods
\cite{HesthavenRozzaEtAl2016,QuarteroniManzoniEtAl2016,Ha14},
using the regularity structure of the solution manifold \cite{Binev2011}
for fast many-query simulation of the model.

%
%

A particular example that we have in mind is the design of printed circuit boards (PCBs).
The design of PCBs has all of the above mentioned properties:
PCBs are nowadays very complex, there is no
scale separation, improvements are often obtained by
local changes of the electronic components and conductive tracks,
and when solved in frequency domain, it is a parameterized problem
with the frequency as a parameter.
A possible change is depicted in Figure
\ref{fig:boards}.
The same applies to integrated circuit (IC) packages, which are in structure similar to PCBs.
Other areas of application could be e.g. resonance analysis in cars or trains
or electromagnetic filter design. 
\begin{figure}
\centering
\includegraphics[width=.9\textwidth,viewport=0 220 1476 470, clip=true]{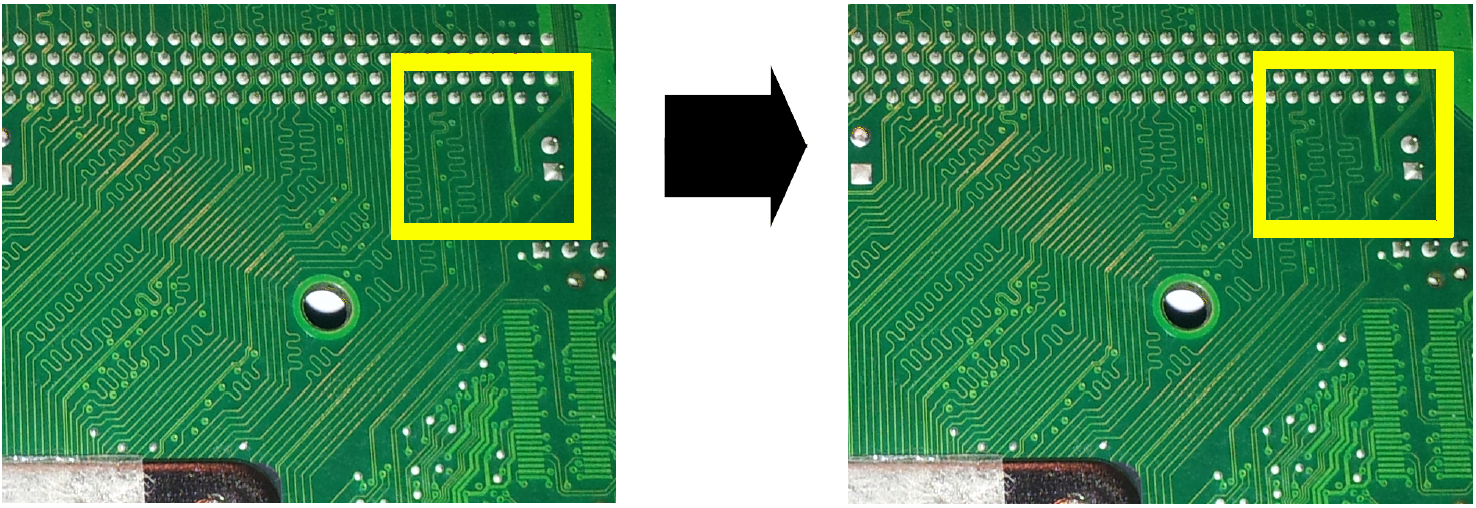}
\caption{DDR memory channel on a printed circuit board subject to local modification of conductive tracks.}
\label{fig:boards}
\end{figure}

ArbiLoMod's goal is the reduction of overall simulation times.
When measuring simulation times, 
we assume that the runtime on a single workstation is not the right
quantity to look at. Instead, the runtime on any
hardware easily available to the user is what matters, which particularly
includes massive computing power in cluster and cloud environments,
but usually not millions of cores as in supercomputers.
As the user still has to pay per compute node in cloud environments,
hundreds to thousands of compute nodes is the foreseen
environment.
Secondary goals during the development of ArbiLoMod were
that the method should be easily implementable
on top of existing finite element schemes and
that the detection of
changed regions vs.\ unchanged regions between two simulation
runs is completely automatic.

After an overview over existing methods in the literature,
the definition of the problem setting and a short overview
of ArbiLoMod, the structure of the paper follows the structure of ArbiLoMod:
In Section \ref{sec:space_decomposition}
the space decomposition used in this paper is presented. 
Training and Greedy algorithms for local basis generation are subject of Section 
\ref{sec:training_and_greedy}. 
The a posteriori error estimator employed 
is discussed Section \ref{sec:a_posteriori}. 
Localized enrichment of the bases 
is described in Section \ref{sec:enrichment}. 
The procedure followed on each geometry change 
is given in Section \ref{sec:onchange}.
Potentials for parallelization are sketched in Section
\ref{sec:runtime_and_communication}. 
Section \ref{sec:numerical_experiments}
contains numerical results.

\subsection*{Existing Approaches}
The combination of ideas of the fields of Reduced Basis Methods,
of multiscale methods and domain decomposition methods
gained a lot of attention in recent time.
In 2002, Maday and R\o nquist
published the ``Reduced Basis Element'' (RBE)
method \cite{Maday2002a,Maday2004},
combining the reduced basis approach with a domain decomposition,
coupling local basis vectors by polynomial Lagrange multipliers on
domain boundaries.
Built on top of the RBE is the ``Reduced Basis Hybrid Method'' (RBHM)
\cite{Iapichino2012} and the ``Discontinuous Galerkin Reduced Basis Element'' method
(DGRBE) \cite{FrancescaAntonietti2015}.
Similar in motivation is the ``Static Condensation Reduced Basis Element''
(SCRBE) method \cite{PhuongHuynh2012}, which also aims at systems
composed of components, where the geometry of the components can
be mapped to reference geometries.
While the connection between the components is simply achieved by
polynomial Lagrange multipliers in the RBE, significant research
has been conducted on choosing the right coupling spaces in
the context of the SCRBE. Choosing the right space at the interfaces
is called ``Port Reduction'' and was included in the name, leading to the
``Port Reduced Static Condensation Reduced Basis Element'' (PR-SCRBE) method
\cite{Eftang2013,Eftang2014} which employs a so-called ``pairwise training''.
A variation of this idea is used in our method.
Recently, also an algorithm to obtain optimal interface spaces
for the PR-SCRBE was proposed \cite{Smetana2015} and a framework
for a posteriori error estimation was introduced \cite{Smetana2015a}.
While the PR-SCRBE performs excellent in using the potentials of cloud
environments, its goals are different from ours.
While being able
to handle various changes to the system simulated,
it does not aim at arbitrary modifications. And it does not try
to hide the localization from the user, but rather exposes it to let
the user decide how the domain decomposition should be done.
%
Another combination of domain decomposition ideas with reduced basis
methods coupling different physical formulations on the domain
boundaries was presented in \cite{Maier2014,Maier2014a}.

A completely different approach to local model order reduction is to
see it as an extension to multiscale methods.
There are two ways to combine RB and multiscale methods.
First is to use RB to accelerate the solution of localized
problems which occur in multiscale methods (``RB within multiscale'').
This has been done by multiple authors, see e.g.
\cite{Abdulle20127014,Abdulle2013203,hesthavenRBMS2015}.
The second way is to use a subspace projection for the global problem,
but use ideas from multiscale methods to construct
the basis functions. This second approach is 
used by ArbiLoMod and is shared with several methods in the
literature:
The ``Generalized Multiscale Finite Element Method'' GMsFEM
\cite{Efendiev2013} uses the idea of
``Multiscale Finite Elements'' MsFEM \cite{Hou97amultiscale} and
constructs reduced spaces which are spanned
by ansatz functions on local patches, using only local
information. It allows for non-fixed
number of ansatz functions on each local patch.
GMsFEM uses local eigenproblems and a partition of unity
for basis generation.
Adaptive enrichment for the GMsFEM is presented by Chung et al.\ in
\cite{Chung2014,Chung2014b} and
online-adaptive enrichment in \cite{Chung2015}.
While an application of the GMsFEM to the problem of
arbitrary local modifications would be very interesting, GMsFEM
was not designed to be communication avoiding. A parallel
implementation of the enrichment described in \cite{Chung2015}
would require the communication of high dimensional basis representations,
which is avoided in ArbiLoMod.

There are also developments of the ``Generalized Finite Element
Method'' GFEM \cite{strouboulis2000design,strouboulis2001generalized}
which could be extended to handle arbitrary local modifications.
Especially the recent development
of the Multiscale-GFEM \cite{babuvska2014machine}
could be promising in this regard.

Similar to ArbiLoMod in spirit is the \! \!
``Localized Reduced Basis Multiscale Method''\!
(LRBMS)
\cite{LRBMSwithOnlineEnrichment,OS15}.
It uses a non overlapping domain decomposition and
discontinuous ansatz spaces, which are coupled using a
DG ansatz at the interface.
While LRBMS could be extended to handle arbitrary local modifications
and it can be implemented in a communication avoiding scheme,
LRBMS cannot easily be implemented on top of an
existing conforming discretization
scheme. ArbiLoMod, in contrast, inherits all conformity properties
from the underlying discretization.
It can be built on top of standard, conforming finite element schemes
and is then conforming itself.
Some of the basic ideas of ArbiLoMod were already published by the authors
\cite{Buhr2009,BO15}. 
First results for electrodynamics were 
published in \cite{BEOR16}.

The main advantage of ArbiLoMod is its speed in cloud environments.
At every design decision during the development of ArbiLoMod, care
was taken to keep the required communication in a parallel implementation
at a minimum. 


\section{Preliminaries}%
\subsection*{Problem Setting}
In this contribution we particularly look at problems that are modeled by 
partial differential equations 
with complex local structure, typically on a very fine scale compared to the overall problem setting. 
In addition, we assume that the problem might depend on a number of parameters $\mu \in \mathcal{P} \subset \R^p$.
As a model problem to explore and design our new simulation technique ArbiLoMod,  we consider 
elliptic equations with complex micro-structure. 
To simplify the presentation, we restrict ourselves to the two dimensional case in this publication.
Thus, let 
$\Omega \subset \R^d$, $d=2$ denote the polygonal
computational domain, and $V$, $H^1_0(\Omega) \subset V \subset H^1(\Omega)$
denote the solution space. We then look at variational problems of the form
\begin{equation}
		{\mathrm{ find }} \ u_\mu \in V \quad {\mathrm{ such\ that }} \qquad a_\mu( u_\mu, v) = \dualpair{f_\mu}{v} \qquad \forall \,v \in V.
\end{equation}
Here, $a_\mu: V \times V \to \R$, $\mu \in \mathcal{P}$ denotes a parameterized
bilinear form with inherent
micro-structure and $f_\mu \in V'$ a force term.
$V$ is equipped with the standard $H^1$ inner product and the thereby
induced norm.
$a_\mu$ is assumed to be coercive and continuous and by
$\alpha_\mu, \gamma_\mu$ we denote the lower (upper) bounds for the coercivity (norm) constants of $a_\mu$, i.e.\ $\alpha_\mu \norm{\varphi}_V^2  \leq a_\mu(\varphi, \varphi)$
for all $\varphi \in V$ and $\norm{a_\mu} \leq \gamma_\mu$.
We further assume that the bilinear form $a_\mu(u,v)$
has a decomposition affine in the parameters and
can be written
as a sum of parameter independent bilinear forms $a^b(u,v)$
with parameter dependent coefficient functions $\theta_b(\mu)$:
\begin{equation}
\label{eq:affine_decomposition}
a_\mu(u,v) = \sum_b \theta_b(\mu) a^b(u,v)
\end{equation}
As an example $a_\mu$ could be given as
\begin{equation}
\label{eq:heat_equation}
		 a_\mu( u, v) = \int_\Omega \sigma_\mu(x) \nabla u(x)  \nabla v(x) {\mathrm{dx}}
\end{equation}
where $\sigma_\mu: \Omega \to \R$ denotes a parameterized heat conduction 
coefficient that varies in space
on a much finer scale then the length scale of $\Omega$.


\subsection*{Structure of ArbiLoMod}
The main ingredients of ArbiLoMod are:
\begin{enumerate}[topsep=0pt,itemsep=-1ex,partopsep=1ex,parsep=1ex]
\item a localizing space decomposition,
\item local training and greedy algorithms,
\item a localized a posteriori error estimator,
\item a localized adaptive enrichment procedure.
\end{enumerate}
ArbiLoMod builds upon a space decomposition of the original ansatz space $V$ consisting
of a set of subspaces $V_i \subset V$ whose direct sum is the
original ansatz space, i.e.
\begin{equation}
V = \bigoplus_i \subs_i.
\end{equation}
The subspaces $V_i$ are meant to have local properties, i.e. all
elements of one subspace $V_i$ have support only in a small subset
of the domain.
In the first step, for each local subspace $V_i$, a reduced
local subspace $\widetilde V_i \subseteq V_i$ is constructed using
training and greedy algorithms.
Thereafter, the global problem is solved in
the space formed by the direct sum of all reduced local subspaces:
\begin{eqnarray}
&&{\mathrm{ with}} \qquad \widetilde V := \bigoplus_i \widetilde V_i\\
\label{eq:red_problem}
&&		{\mathrm{ find }} \qquad  \widetilde u_\mu \in \widetilde V \quad {\mathrm{ such\ that }} \qquad a_\mu( \widetilde u_\mu, v) = \dualpair{f_\mu}{v} \qquad \forall \,v \in \widetilde V.
\nonumber
\end{eqnarray}
To assess the quality of the thus obtained solution,
 a localized a posteriori error estimator is employed.
If necessary, the solution is improved by enriching the reduced
local subspaces, using a residual based, localized enrichment procedure.
Finally, on each localized change, affected bases are discarded 
and the procedure starts over.
An overview is given in Figure~\ref{fig:workflow}.
\begin{figure}
\centering
\includegraphics[width=.7\textwidth]{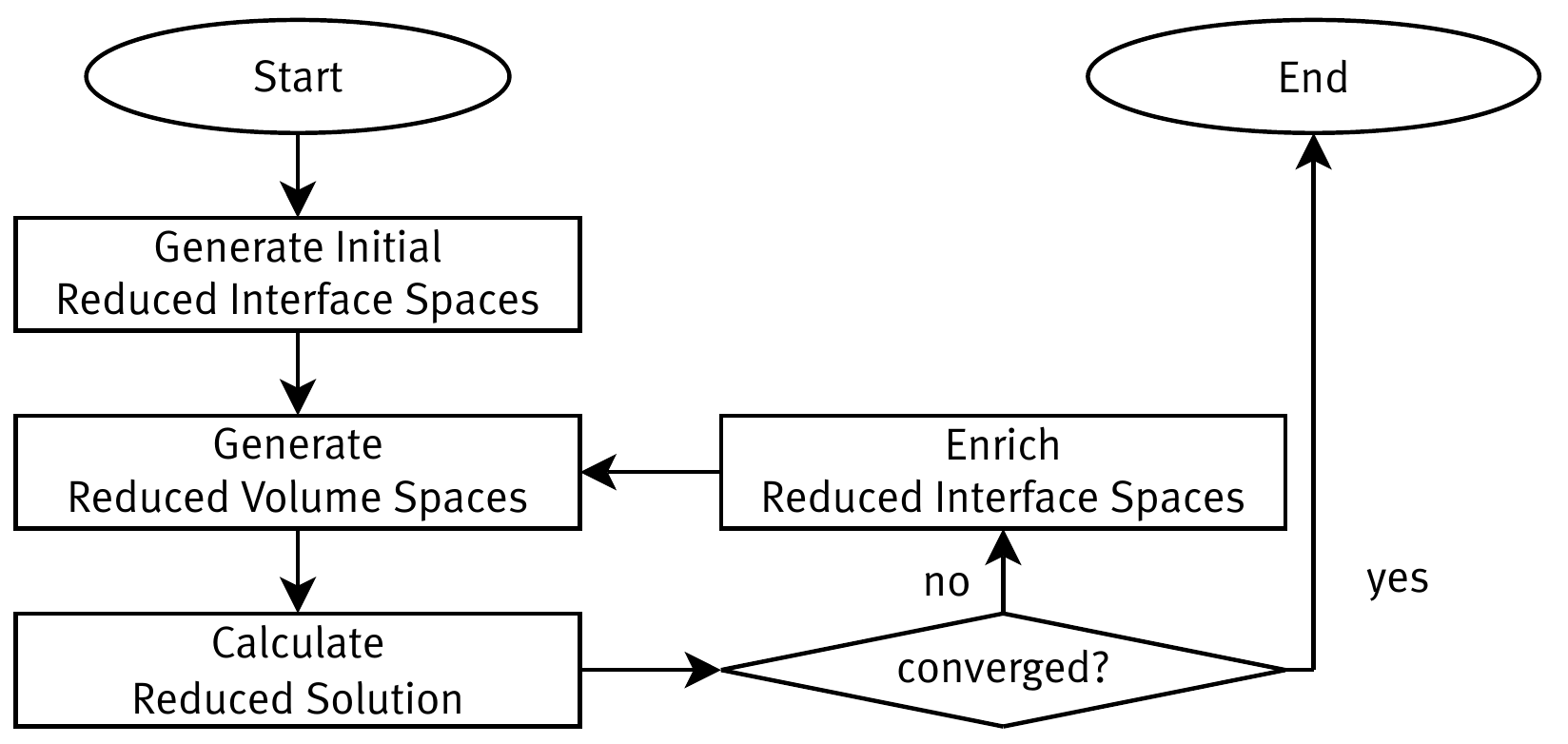}
\caption{Overview of ArbiLoMod.
Generation of initial reduced interface spaces by training
and of reduced volume spaces by greedy basis generation is subject of Section \ref{sec:training_and_greedy}.
Convergence is assessed with the localized a posteriori error estimator presented in Section \ref{sec:a_posteriori}.
Enrichment is discussed in Section \ref{sec:enrichment}.
On each geometry change, the procedure starts over,
where new initial interface spaces are only generated in the region
affected by the change.
}
\label{fig:workflow}
\end{figure}

\section{Space Decomposition}
\label{sec:space_decomposition}
In this section we introduce definitions for the different kinds of
spaces needed in the method.
Since ArbiLoMod is designed to work on top of existing discretization schemes,
we formulate the method on a discrete level.
The following definitions are for two dimensional problems, but
can be easily extended for three dimensional problems.
\subsection{Basic Subspaces}
$V_h$ denotes a discrete ansatz space, spanned by ansatz functions
$\{\psi_i\}_{i=1}^N =: \basis$.
We assume that
the ansatz functions have a localized support, which is true
for many classes of ansatz functions like Lagrange- or
N\'ed\'elec-type functions.
We first introduce a direct decomposition of the ansatz
space $V_h$ into ``basic subspaces''. These are used to construct
the space decomposition later.
To obtain the subspaces we classify the
ansatz functions by their support and define each subspace as the
span of all ansatz functions of one class.
To this end, we introduce a non overlapping domain decomposition of the
original domain $\Omega$ into open subdomains $\Omega_i$:
\begin{equation}
\overline \Omega = \bigcup_{i=1}^{N_D}\overline \Omega_i  \qquad \Omega_i \cap \Omega_j = \emptyset\ \mathrm{for}\ i\ne j
\label{eq:dd}
\end{equation}
where $N_D$ is the number of subdomains.
For each $\psi$ in $\basis$, we call $\domains_\psi$
the set
of indices of subdomains that have non-empty intersection with the support of $\psi$, i.e.
\begin{equation}
\domains_\psi :=
\Big\{i \in \{1, \dots, N_D\} \ \Big| \ \supp(\psi) \cap \Omega_i \ne \emptyset \Big\}.
\end{equation}
We collect all occurring domain sets in $\Upsilon$:
\begin{equation}
\Upsilon := \Big\{ \domains_\psi , \psi \in \basis \Big\}.
\end{equation}
We define sets for each
codimension:
\begin{equation}\label{def:codimsets}
\Upsilon_0 := \Big\{ \xi \in \Upsilon \ \Big| \ |\xi| = 1\Big\},\
\Upsilon_1 := \Big\{ \xi \in \Upsilon \ \Big| \ |\xi| = 2\Big\},\
\Upsilon_2 := \Big\{ \xi \in \Upsilon \ \Big| \ |\xi| > 2\Big\}.
\end{equation}
The elements of $\Upsilon_2$,$\Upsilon_1$, and $\Upsilon_0$ can be associated with interior vertices
(codimension 2), faces (codimension 1) and
cells (codimension 0) of the domain decomposition.
The classification is very similar
to the classification of mesh nodes in domain decomposition methods, see
for example
\cite[Def. 3.1]{klawonn2006dual}
or \cite[Def. 4.2]{toselli2005domain}. 
\begin{definition}[Basic Subspaces]
\label{def:basic_space_def}
For each element $\xi \in \Upsilon$  we define a basic subspace $\subsbd_\xi$ of $V$ as:
\begin{equation}
\subsbd_\xi := \spanset\Big\{\psi \in \basis \ \Big| \ \domains_\psi = \xi \Big\}.
\nonumber
\end{equation}
\end{definition}
\begin{remark}[Basic Decomposition]
The definition of $\subsbd_\xi$ induces a direct decomposition of $V_h$:
\begin{equation}
V_h = \bigoplus_{\xi \in \Upsilon} \subsbd_\xi.
\nonumber
\end{equation}
\end{remark}
\subsection{Space Decomposition}
\label{sec:decomposition}
As mentioned in the introduction, ArbiLoMod is based on a space
decomposition. It can work on the Basic Decomposition
introduced in Definition \ref{def:basic_space_def}. However,
faster convergence and smaller basis sizes are achieved using the
modified space decomposition introduced in this section.
Here we assume that the discrete ansatz space $V_h$ is spanned
by finite element ansatz functions defined on a mesh which 
resolves the subdomains $\Omega_i$.

For each of the spaces $\subsbd_\xi$ defined in Definition \ref{def:basic_space_def},
we calculate extensions. The extensions are computed on the
``extension space'' $\extension{\subsbd_\xi}$ which is defined as
\begin{equation}
\label{eq:definition_extension_space}
\extension{\subsbd_\xi} := \bigoplus \Big\{\subsbd_\zeta \ \Big| \ \zeta \subseteq \xi \Big\}.
\end{equation}
\begin{figure}
\footnotesize
\centering
\def\myscale{1.2}
\def\circsize{0.06}
\def\circsizetwo{0.09}
\begin{tikzpicture}[scale=\myscale]
\node at (0.5+0.125,0.5) [color=Xred] {\Huge 1};
\draw [step=2.5mm,color=gray,very thin] (-0.15,-0.15) grid (1.15,1.15);
\draw [step=1cm,black,line width=0.7mm] (-0.15,-0.15) grid (1.15,1.15);
\foreach \x in {1, ..., 3} {
  \foreach \y in {1, ..., 3} {
    \draw (\x * 0.25, \y * 0.25) circle(\circsizetwo);
  }
}
\foreach \x in {1, ..., 3} {
  \foreach \y in {1, ..., 3} {
    \draw [fill=black] (\x * 0.25, \y * 0.25) circle(\circsize);
  }
}
\node at (0.5,-0.5) {$\xi = \{1\} \in \Upsilon_0$};
\end{tikzpicture}
\begin{tikzpicture}[scale=\myscale]
\node at (0.5+0.125,0.5) [color=Xred] {\Huge 1};
\node at (1.5+0.125,0.5) [color=Xred] {\Huge 2};
\draw [step=2.5mm,color=gray,very thin] (-0.15,-0.15) grid (2.15,1.15);
\draw [step=1cm,black,line width=0.7mm] (-0.15,-0.15) grid (2.15,1.15);
\foreach \x in {1, ..., 7} {
  \foreach \y in {1, ..., 3} {
    \draw (\x * 0.25, \y * 0.25) circle(\circsizetwo);
  }
}
\foreach \x in {4} {
  \foreach \y in {1, ..., 3} {
    \draw [fill=black] (\x * 0.25, \y * 0.25) circle(\circsize);
  }
}
\node at (1,-0.5) {$\xi = \{1,2\} \in \Upsilon_1$};
\end{tikzpicture}
\begin{tikzpicture}[scale=\myscale]
\node at (0.5+0.125,1.5) [color=Xred] {\Huge 1};
\node at (1.5+0.125,1.5) [color=Xred] {\Huge 2};
\node at (0.5+0.125,0.5) [color=Xred] {\Huge 3};
\node at (1.5+0.125,0.5) [color=Xred] {\Huge 4};
\draw [step=2.5mm,color=gray,very thin] (-0.15,-0.15) grid (2.15,2.15);
\draw [step=1cm,black,line width=0.7mm] (-0.15,-0.15) grid (2.15,2.15);
\foreach \x in {1, ..., 7} {
  \foreach \y in {1, ..., 7} {
    \draw (\x * 0.25, \y * 0.25) circle(\circsizetwo);
  }
}
\foreach \x in {4} {
  \foreach \y in {4} {
    \draw [fill=black] (\x * 0.25, \y * 0.25) circle(\circsize);
  }
}
\node at (1,-0.5) {$\xi = \{1,2,3,4\} \in \Upsilon_2$};
\end{tikzpicture}
\begin{tikzpicture}[scale=\myscale]
\draw [fill=black] (0.15,1.2) circle(\circsize);
\node at (0.4,1.2) [right] {dof of $\subsbd_\xi$};
\draw (0.15,0.9) circle(\circsizetwo);
\node at (0.4,0.9) [right] {dof of $\extension{\subsbd_\xi}$};
\draw [color=gray, very thin] (0,0.6) -- (0.3,0.6);
\node at (0.4,0.6) [right] {mesh line};
\draw [line width=0.7mm] (0,0.3) -- (0.3,0.3);
\node at (0.4,0.3) [right] {domain boundary};
\node at (0.15,0) [color=Xred] {\huge 1};
\node at (0.4, 0) [right] {domain number};
\node at (0,-1) {};
\end{tikzpicture}
\caption{Visualization of basic spaces $U_\xi$ and their extension spaces for $Q^1$ ansatz functions
(one dof per mesh node).}
\label{fig:extension_space}
\end{figure}
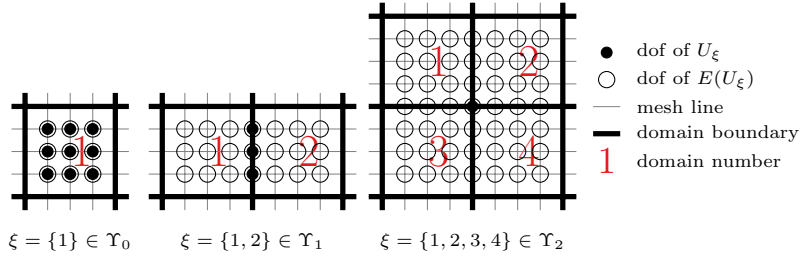

Examples for extension spaces are given in Figure~\ref{fig:extension_space}.
For each space $\subsbd_\xi$, a linear extension operator $\extend$ is defined:
\begin{equation}
\extend : \subsbd_\xi \rightarrow \extension{\subsbd_\xi}.
\end{equation}
For all $\xi$ in $\Upsilon_0$, $\extend$ is just the identity.
For all $\xi$ in $\Upsilon_1$, we extend by solving the
homogeneous version of the equation with Dirichlet zero
boundary values for one (arbitrary) chosen
$\defaultparameter \in \parspace$.
For example, in the situation depicted in Fig. \ref{fig:extended_base_1},
\begin{eqnarray*}
\extend : \subsbd_{\{1,2\}} &\rightarrow& \subsbd_{\{1\}} \oplus \subsbd_{\{1,2\}} \oplus \subsbd_{\{2\}}\\
\varphi  &\mapsto& \varphi + \psi\\
&&\mbox{where } \psi \in \subsbd_{\{1\}} \oplus \subsbd_{\{2\}} \mbox{ solves}\\
&&
 a_{\defaultparameter}(\varphi + \psi, \phi) = 0 \qquad \forall \phi \in \subsbd_{\{1\}} \oplus \subsbd_{\{2\}}
\end{eqnarray*}
For all $\xi$ in $\Upsilon_2$, $\extend$ is defined by first
extending linearly to zero on all edges in the extension domain, i.e.~in all spaces in $\extension{\subsbd_\xi}$
which belong to $\Upsilon_1$. Then, in a second step, the homogeneous version
of the equation with Dirichlet boundary values is solved on the spaces in
$\extension{\subsbd_\xi}$ which belong to $\Upsilon_0$.
The procedure is visualized in Figure~\ref{fig:vertexextension}.
The functions constructed by this two step procedure are
the same as the MsFEM basis functions used by Hou and Wu \cite{Hou97amultiscale}.
Note that the base functions for $\extension{\subsbd_\xi}, \xi \in \Upsilon_2$ form a partition 
of unity in the interior of the coarse partition of the domain. They can be completed to form 
a partition of unity on the whole coarse partition of the domain if suitable base functions for 
the vertices at the boundary of the domain are added. 
This will be used in Section \ref{sec:a_posteriori} below for 
the robust and efficient localization of an a posteriori error estimator. 
Examples of extended functions for all codimensions are given in Figure~\ref{fig:extended_base}.
In the case of the Laplace equation these basis functions coincide with the hat functions on the 
coarse partition (see Figure~\ref{fig:vertexextension}).  
As discussed in Section \ref{sec:a_posteriori} below, the choice of hat function can be an alternative 
choice that allows for a better a priori bound of the constants in the localized a posteriori error estimator, 
as their gradient is controled by $1/H$ -- where $H$ denotes the mesh size of the macro partition -- independent of the contrast of the data.
\newlength{\foobar}
\setlength{\foobar}{2cm}
\def\myscale{1.2}
\def\circsize{0.06}
\def\circsizetwo{0.08}
\begin{figure}
\footnotesize
\centering
\begin{subfigure}[b]{0.19\textwidth}
\centering
\begin{tikzpicture}[scale=\myscale]
\draw [step=0.4mm,color=gray,very thin] (-0.15,-0.15) grid (1.15,1.15);
\draw [step=1cm,black,line width=0.7mm] (-0.15,-0.15) grid (1.15,1.15);
\node at (0.5,0.5) {\includegraphics[width=\myscale cm]{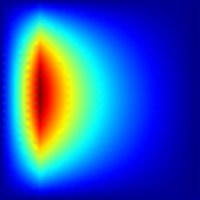}};
\end{tikzpicture}
\caption{$\xi = \{1\} \in \Upsilon_0$}
\label{fig:extended_base_0}
\end{subfigure}
\begin{subfigure}[b]{0.24\textwidth}
\centering
\begin{tikzpicture}[scale=\myscale]
\draw [step=0.4mm,color=gray,very thin] (-0.15,-0.15) grid (2.15,1.15);
\draw [step=1cm,black,line width=0.7mm] (-0.15,-0.15) grid (2.15,1.15);
\node at (1,0.5) {\includegraphics[width=\myscale\foobar]{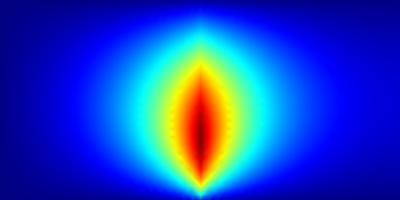}};
\end{tikzpicture}
\caption{$\xi = \{1,2\} \in \Upsilon_1$}
\label{fig:extended_base_1}
\end{subfigure}
\begin{subfigure}[b]{0.26\textwidth}
\centering
\begin{tikzpicture}[scale=\myscale]
\draw [step=0.4mm,color=gray,very thin] (-0.15,-0.15) grid (2.15,2.15);
\draw [step=1cm,black,line width=0.7mm] (-0.15,-0.15) grid (2.15,2.15);
\node at (1,1) {\includegraphics[width=\myscale\foobar]{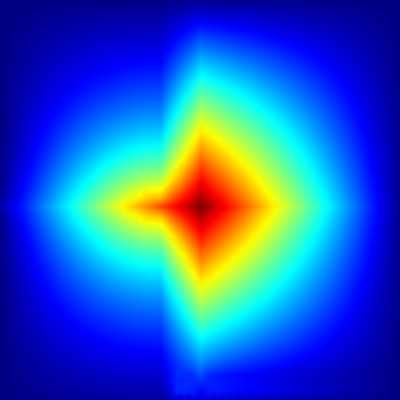}};
\end{tikzpicture}
\caption{$\xi = \{1,2,3,4\} \in \Upsilon_2$}
\label{fig:extended_base_2}
\end{subfigure}
\begin{subfigure}[b]{0.26\textwidth}
\centering
\begin{tikzpicture}[scale=\myscale]
\draw [line width=0.6mm] (0,0) -- (0.3,0) node[right]{domain \!boundary};
\draw [color=gray, very thin] (0,0.3) -- (0.3,0.3) node[right,color=black]{mesh line};
\node at (0,-1) {};
\end{tikzpicture}
\end{subfigure}
\caption{
Visualization of some example elements of the local subspaces $\subssd_\xi$
for inhomogeneous coefficients.
The structure in the solution results from variations in the
heat conduction coefficient.
}
\label{fig:extended_base}
\end{figure}
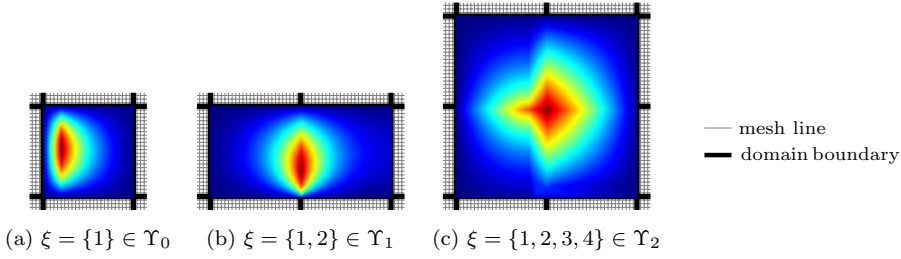

\begin{figure}
\centering
\begin{subfigure}[b]{0.3\textwidth}
\centering
\includegraphics[width=\textwidth]{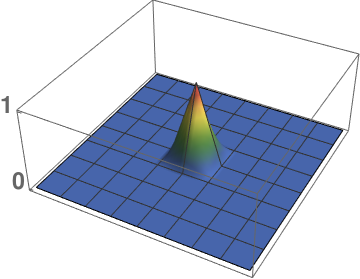}
\caption{Value 1 in $U_{\{1,2,3,4\}}$}
\label{fig:vertexextension1}
\end{subfigure}
\begin{subfigure}[b]{0.3\textwidth}
\centering
\includegraphics[width=\textwidth]{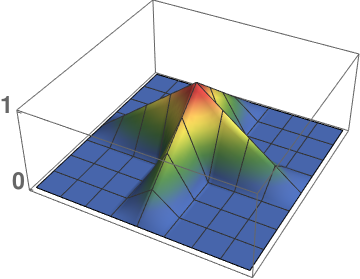}
\caption{Linear in $U_\xi, \ \xi \in \Upsilon_1$}
\label{fig:vertexextension2}
\end{subfigure}
\begin{subfigure}[b]{0.3\textwidth}
\centering
\includegraphics[width=\textwidth]{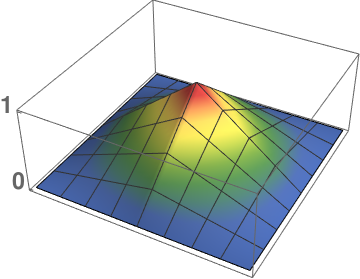}
\caption{Solving in $U_\xi, \ \xi \in \Upsilon_0$}
\label{fig:vertexextension3}
\end{subfigure}
\caption{
$\mathrm{Extend}$ operator is executed in two steps for spaces $V_\xi, \xi \in \Upsilon_2$. 
Exemplified for the generation of $V_{\{1,2,3,4\}}$ from $U_{\{1,2,3,4\}}$. Mesh and spaces as depicted in Fig \ref{fig:extension_space}.
Script: \texttt{generate\_vertex\_extension\_plots.nb}.
}
\label{fig:vertexextension}
\end{figure}

For the communication avoiding properties of the ArbiLoMod, it is
important to note that extensions can be calculated independently on
each domain, i.e. $\extend(\varphi)|_{\Omega_i}$ can be calculated
having only information about $\varphi$ and $\Omega_i$, without
knowledge about other domains.
Using this operator, we define the local subspaces $\subssd_\xi$:

\begin{definition}[Extended Subspaces]
\label{def:extended_space_def}
For each element $\xi \in \Upsilon$  we define an extended subspace $\subssd_\xi$ of $V_h$ as:
\begin{equation}
\label{eq:space_decomposition}
\subssd_\xi := \Big\{ \extend (\varphi) \ \Big| \ \varphi \in \subsbd_\xi \Big\}.
\nonumber
\end{equation}
According to the definition in \eqref{def:codimsets} we call $\subssd_\xi $ a cell, face, or vertex space,
if $\xi \in \Upsilon_0, \xi \in \Upsilon_1$, or $\xi \in \Upsilon_2$, respectively (cf.\ Fig. \ref{fig:extended_base}).
\end{definition}
\begin{remark}[Extended Decomposition]
The definition of $\subssd_\xi$ induces a direct decomposition of $V_h$:
\begin{equation}
V_h = \bigoplus_{\xi \in \Upsilon} \subssd_\xi.
\nonumber
\end{equation}
\end{remark}
~\\[-1ex]
\noindent Space decompositions of the same spirit are used
in the context of Component Mode Synthesis
(CMS), see \cite{Hetmaniuk2010,hetmaniuk2014error}.
\subsection{Projections}
\label{sec:projection_operators}

\begin{definition}[Local Projection Operators]
\label{def:local_projection_operators}
The projection operators
$\subsmap_{\subsbd_\xi} : V_h \rightarrow \subsbd_\xi$
and
$\subsmap_{\subssd_\xi} : V_h \rightarrow \subssd_\xi$
are defined by the relation
\begin{equation}
\varphi =
\sum_{\xi \in \Upsilon} \subsmap_{\subsbd_\xi}(\varphi)
= \sum_{\xi \in \Upsilon} \subsmap_{\subssd_\xi}(\varphi)
\qquad \forall \varphi \in V_h. \nonumber
\end{equation}
As both the subspaces $\subsbd_\xi$ and the subspaces $\subssd_\xi$ form
a direct decomposition of the space $V_h$, the projection operators are
uniquely defined by this relation.
\end{definition}

The implementation of the projection operators
$\subsmap_{\subsbd_\xi}$ is very easy: It is just extracting
the coefficients of the basis functions forming $\subsbd_\xi$ out
of the global coefficient vector.
The implementation of the projection operators $\subsmap_{\subssd_\xi}$
is more complicated and involves the solution of local problems,
see Algorithm \ref{algo:projections}.
\begin{algorithm2e}
\DontPrintSemicolon
\SetAlgoVlined
\SetKwFunction{SpaceDecomposition}{SpaceDecomposition}
\SetKwInOut{Input}{Input}
\SetKwInOut{Output}{Output}
\Fn{\SpaceDecomposition{$\varphi$}}{
  \Input{function $\varphi \in V_h$}
  \Output{decomposition of $\varphi$}
  \tcc{iterate over all codimensions in decreasing order}
  \For{
    $\mathrm{codim} \in \{d, \dots, 0\}$
  }%
  {
    \For{
        $\xi \in \Upsilon_\mathrm{codim}$
    }
    {
        $\varphi_\xi \leftarrow \mathrm{Extend}(\subsmap_{\subsbd_\xi}(\varphi))$\;
        $\varphi \leftarrow \varphi - \varphi_\xi$
    }
  }
  \Return $\{\varphi_\xi\}$\;
}
\caption{Projections in $\subssd_\xi$}
\label{algo:projections}
\end{algorithm2e}



\section{Local Basis Generation}
\label{sec:training_and_greedy}
For each local subspace $\subssd_\xi$, an initial reduced local
subspace $\rsubssd_\xi \subseteq \subssd_\xi$ is generated, using
only local information from an environment around the support of the
elements in $\subsbd_\xi$. The strategy used to construct
these
reduced local subspaces depends on the type of the
space: whether $\xi$ belongs to $\Upsilon_0$, $\Upsilon_1$ or
$\Upsilon_2$. The three strategies are given in the
following.
The local basis generation algorithms
can be run in parallel, completely independent of each other.
See Section \ref{sec:runtime_and_communication} for further
discussion of the potential parallelization.
As the algorithms only use local information, their results do not change
when the problem definition is changed outside of the area they took into
account. So there is no need to rerun the algorithms in this case.
Our numerical results indicate that the spaces obtained by
local trainings and greedys have good approximation properties
(see Section \ref{sec:numerical_experiments}).
The quality of the obtained solution will be guaranteed by
the a posteriori
error estimator presented in Section \ref{sec:a_posteriori} below.
\subsection{Basis Construction for Reduced Vertex Spaces}
\label{sec:codim_2_spaces}
The spaces $\subssd_\xi$ for $\xi \in \Upsilon_2$ are spanned by
only one function (see Figure~\ref{fig:extended_base_2} for an example)
and are thus one dimensional. The reduced spaces
are therefore chosen to coincide with the original space, i.e.
$
\widetilde \subssd_\xi := \subssd_\xi, \forall \xi \in \Upsilon_2.
$
\subsection{Local Training for Basis Construction of Reduced Face Spaces}
\label{sec:codim_n_training}
To generate an initial reduced local subspace
for $\subssd_\xi, \ \xi \in \Upsilon_1$
we use a local training procedure.
Its main four steps are to
\begin{enumerate}[topsep=0pt,itemsep=-1ex,partopsep=1ex,parsep=1ex]
\item solve the equation on a small domain around the space in question
with zero boundary values for all parameters in the training set $\Xi$,
\item
solve the homogeneous equation
repeatedly on a small domain around the space in question with
random boundary values for all parameters in $\Xi$,
\item
apply the space decomposition to all obtained local solutions
to obtain the part belonging to the space in question
and
\item
use a greedy procedure to create a space approximating this set.
\end{enumerate}
The complete algorithm is given in Algorithm \ref{algo:training}
and explained below.

The training is inspired by the ``Empirical Port Reduction'' introduced in
Eftang et al.\ \cite{Eftang2013} but differs in some key points.
The main differences are:
(1)~%
Within~\cite{Eftang2013}, the trace of solutions at the interface to be trained is used.
This leads to the requirement that interfaces between domains do not intersect.
In ArbiLoMod, a space decomposition is used instead. This allows ports to intersect,
which in turn allows the decomposition of space into domains.
(2)~%
The ``Empirical Port Reduction'' trains with a pair of domains. We use an environment
of the interface in question, which contains six domains in the 2D case.
In 3D, it contains 18 domains.
(3)~%
PR-SCRBE aims at providing a library of domains
which can be connected at their interfaces. The reduced interface spaces
are used in different domain configurations and have to be valid in all of them.
Within the context of ArbiLoMod, no database of domains is created and 
the interface space is constructed only for the configuration at hand,
which simplifies the procedure.
(4)~%
The random boundary values used in \cite{Eftang2013} are generalized
Legendre polynomials with random coefficients.
In ArbiLoMod, the finite element basis functions with random coefficients
are used, which simplifies the construction greatly, especially when there
is complex structure within the interface.

\subsubsection*{The Training Space}
\begin{figure}
\centering
\def\myscale{1.3}
\def\circsize{0.06}
\def\circsizetwo{0.09}
\begin{tikzpicture}[scale=\myscale]
\node at (0.5+0.125,1.5) [color=Xred] {\Huge 1};
\node at (1.5+0.125,1.5) [color=Xred] {\Huge 2};
\node at (2.5+0.125,1.5) [color=Xred] {\Huge 3};
\node at (0.5+0.125,0.5) [color=Xred] {\Huge 4};
\node at (1.5+0.125,0.5) [color=Xred] {\Huge 5};
\node at (2.5+0.125,0.5) [color=Xred] {\Huge 6};
\draw [step=2.5mm,color=gray,very thin] (-0.2,-0.2) grid (3.2,2.2);
\draw [step=1cm,black,line width=0.7mm] (-0.2,-0.2) grid (3.2,2.2);
\foreach \x in {1, ..., 11} {
  \foreach \y in {1, ..., 7} {
    \draw (\x * 0.25, \y * 0.25) circle(\circsizetwo);
  }
}
\foreach \x in {5, ..., 7} {
  \foreach \y in {4} {
    \draw [fill=black] (\x * 0.25, \y * 0.25) circle(\circsize);
  }
}
\foreach \x in {1, ..., 11} {
  \foreach \y in {0,8} {
    \draw [Xred,fill=Xred!60] (\x * 0.25, \y * 0.25) circle(\circsize);
  }
}
\foreach \x in {0,12} {
  \foreach \y in {0, ..., 8} {
    \draw [Xred,fill=Xred!60] (\x * 0.25, \y * 0.25) circle(\circsize);
  }
}
\node at (1.5,-0.5) {$\xi = \{2,5\} \in \Upsilon_1$};
\end{tikzpicture}
\begin{tikzpicture}[scale=\myscale]
\draw [fill=black] (0.15,1.5) circle(\circsize);
\node at (0.4,1.5) [right] {dof of $\subsbd_\xi$};
\draw (0.15,1.2) circle(\circsizetwo);
\node at (0.4,1.2) [right] {dof of $\training{\subsbd_\xi}$};
\draw (0.15,0.9) [Xred!,fill=Xred!60] circle(\circsize);
\node at (0.4,0.9) [right]{dof of $\coupling{\training{\subsbd_\xi}}$};
\draw [color=gray, very thin] (0,0.6) -- (0.3,0.6);
\node at (0.4,0.6) [right] {mesh line};
\draw [line width=0.7mm] (0,0.3) -- (0.3,0.3);
\node at (0.4,0.3) [right] {domain boundary};
\node at (0.15,0) [color=Xred] {\huge 1};
\node at (0.4, 0) [right] {domain number};
\node at (0,-.9) {};
\end{tikzpicture}
\caption{Visualization of basic spaces $U_{\{2,5\}}$, its training space, and the coupling space of its training space for $Q^1$ ansatz functions
(one dof per mesh node).}
\label{fig:training_space}
\end{figure}
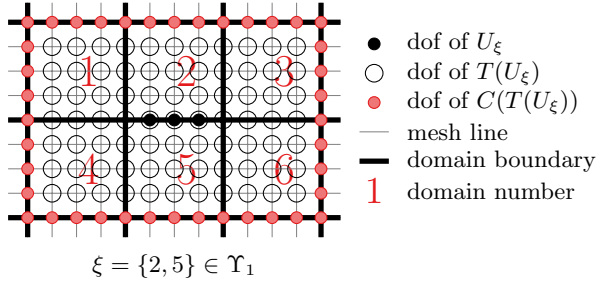

To train a basis for $\widetilde \subssd_\xi$, $\xi \in \Upsilon_1$, 
we start
from the subspace $\subsbd_\xi$.
For each subspace $\subsbd_\xi$, we define a corresponding training
space $\training{\subsbd_\xi}$ on an environment
associated with the face $\xi$.
The following definition is geometric
and tailored to a domain decomposition in rectangular domains. 
More complex domain decompositions would need a more complex definition here.
We define the neighborhood
\begin{equation}
\neigh_\xi := \Big\{i \in \{1, \dots, N_D\} \ \Big| \ \overline{\Omega}_i \cap \Bigl( \bigcap_{k \in \xi} \overline{\Omega}_k \Bigr) \ne \emptyset \Big\}
\end{equation}
and with that the training spaces
\begin{equation}
\training{\subsbd_\xi} := \bigoplus \Big\{\subsbd_\zeta \ \Big| \ \zeta \subseteq \neigh_\xi \Big\}.
\end{equation}
The training space is coupled to the rest of the system via its
coupling space
\begin{equation}
\coupling{\training{\subsbd_\xi}}
:=
\bigoplus \Big\{\subsbd_\zeta \ \Big| \ \zeta \cap \neigh_\xi \ne \emptyset , \zeta \nsubseteq \neigh_\xi \Big\}.
\end{equation}
A sketch of the degrees of freedom associated with the respective spaces is given in Figure~\ref{fig:training_space}.
These definitions are also suitable in the 3D case.
We have fixed the size of the neighborhood to one domain from the interface in question in each direction.
This facilitates the setup of local problems and the 
handling of local changes: After a local change, the affected domains are determined.
Afterwards, all trainings have to be redone 
for those spaces which contain an affected domain in their training domain.
While a larger or smaller training domain might be desirable in some cases 
(see \cite{henning_oversampling}), it is not
necessary:
As missing global information is added in the enrichment step,
ArbiLoMod always converges to the desired accuracy,
even if the training domain is not of optimal size. So the advantages of having
the size of the training domain fixed to one domain outweighs its drawbacks.

The reduced basis must be rich enough to
handle two types of right hand sides up to a given accuracy $\varepsilon_\mathrm{train}$:
\begin{inparaenum}[(a)]
\item source terms and boundary conditions, and
\item arbitrary values on the coupling interface,
\end{inparaenum}
both in the whole parameter space $\parspace$.
We define an extended parameter
space $\parspace \times \coupling{\training{\subsbd_\xi}}$. For this
parameter space we construct a training space
$\Xi \times G \subset \parspace \times \coupling{\training{\subsbd_\xi}}$, where $G$
denotes an appropriate sampling of
$\coupling{\training{\subsbd_\xi}}$.
We use the finite element
basis $\basis_{\coupling{\training{\subsbd_\xi}}}$ on the coupling space and
generates $M$ random coefficient vectors $r_i$ of size $N_\basis =
\dim({\coupling{\training{\subsbd_\xi}}})$. With this an
individual coupling space function $\varphi \in
\coupling{\training{\subsbd_\xi}}$ is constructed
as
\begin{equation}
\varphi_i = \sum_{j=1}^{N_\basis} r_{ij} \phi_j; \qquad\phi_j \in \basis_{\coupling{\training{\subsbd_\xi}}}.\label{eq:3}
\end{equation}
For our numerical experiments in Section \ref{sec:numerical_experiments},
we use uniformly distributed random coefficients over the
interval $[-1,1]$ and Lagrange basis functions.
For each $\mu \in \Xi$ and each pair $(\mu,g_c) \in \Xi \times G$
we construct snapshots $u_f$ and $u_c$ as solutions for
right hand sides $f_\mu(.)$, $a_{\mu}(g_c,.)$ respectively, i.e.
\begin{align}
  a_{\mu}(u_f,\phi) &= \dualpair{f_\mu}{\phi} \qquad \forall \phi \in \training{\subsbd_\xi},\\
  a_{\mu}(u_c,\phi) & = - a_{\mu}(g_c,\phi) \quad \forall \phi \in \training{\subsbd_{\xi}}\,.
  \nonumber
\end{align}
Based on the set of snapshots which we call $Z$, a reduced basis $\basis$
is constructed using a greedy algorithm (cf.\ Alg.\ \ref{algo:snapshot_greedy}).
In the numerical experiments, the $V$-norm and $V$-inner product are used.
The complete generation of the reduced face spaces $\rsubssd_\xi, \ \xi \in \Upsilon_1$ with basis $\rbasis_{\rsubssd_\xi}$
is summarized in Alg.\ \ref{algo:training}.
\begin{algorithm2e}
\DontPrintSemicolon
\SetAlgoVlined
\SetKwFunction{SnapshotGreedy}{SnapshotGreedy}
\SetKwInOut{Input}{Input}
\SetKwInOut{Output}{Output}
\Fn{\SnapshotGreedy{$Z,\varepsilon_\mathrm{train}$}}{
  \Input{set of elements to approximate $Z$,\\ training tolerance $\varepsilon_\mathrm{train}$}
  \Output{basis of approximation space $\basis$}
  $\basis \leftarrow \emptyset$ \;
  \While{$\max_{z \in Z} \norm{z} > \varepsilon_\mathrm{train}$}{
    $\hat z \leftarrow \argmax_{z \in Z} \norm{z}$\;
    $\hat z \leftarrow \frac{\hat z}{\norm{\hat z}}$\;
    $Z \leftarrow \left\{z - (z,\hat z)\hat z \ | \ z \in Z \right\}$\;
    $\basis \leftarrow \basis \cup \{\hat z\}$ \;
  }
  \Return $\basis$\;
}
\caption{SnapshotGreedy}
\label{algo:snapshot_greedy}
\end{algorithm2e}
\begin{algorithm2e}
\DontPrintSemicolon
\SetAlgoVlined
\SetKwFunction{Training}{Training}
\SetKwFunction{RandomSampling}{RandomSampling}
\SetKwFunction{SnapshotGreedy}{SnapshotGreedy}
\SetKwInOut{Input}{Input}
\SetKwInOut{Output}{Output}
\Fn{\Training{$\xi, M, \varepsilon_\mathrm{train}$}}{
  \Input{space identifier $\xi$,\\ number of random samples $M$,\\ training tolerance $\varepsilon_\mathrm{train}$}
  \Output{reduced local subspace $\rsubssd_\xi$}
  $G \leftarrow$ \RandomSampling{$\xi$, M}\;
  $Z \leftarrow \emptyset$\;
  \ForEach{$\mu \in \Xi$}{
    \Find{$u_f \in \training{\subsbd_\xi}$ \st}{
      $
      \qquad a_{\mu}(u_f,\phi) = f_\mu(\phi) \qquad \forall \phi \in \training{\subsbd_\xi}
      $}
    $Z \leftarrow Z \cup \subsmap_{\subssd_\xi}(u_f)$\;
    \ForEach{$g_c \in G$}{
      \Find{$u_c \in \training{\subsbd_{\xi}}$ \st}{
        $
        \qquad a_{\mu}(u_c + g_c,\phi) = 0 \qquad \forall \phi \in \training{\subsbd_{\xi}}
        $}
      $Z \leftarrow Z \cup \subsmap_{\subssd_\xi}(u_c)$\;
    }
  }
  $\rbasis_{\subssd_\xi} \leftarrow$ \SnapshotGreedy{$Z, \varepsilon_\mathrm{train}$}\;
  \Return $\spanset(\rbasis_{V_{\xi}})$\;
}
\caption{Training to construct reduced face spaces $\rsubssd_\xi, \ \xi \in \Upsilon_1$}
\label{algo:training}
\end{algorithm2e}
\subsection{Basis Construction for Reduced Cell Spaces
Using Local Greedy}
\label{sec:codim_0_greedy}
For each cell space $\subssd_\xi, \ \xi \in \Upsilon_0$ we create a reduced space
$\rsubssd _\xi$. These spaces should be able to
approximate the solution in the associated part of the space decomposition for any
variation of functions from reduced vertex or face spaces that are coupled with it.
We define the reduced coupling space $\rcoupling{\subssd_{\xi}}$
and its basis $\rbasis_{\rcoupling{\subssd_\xi}}$.
\begin{equation}
\Upsilon^C_\xi := \Big\{\zeta \in \Upsilon \ | \ \zeta \cap \xi \ne \emptyset,
\zeta \nsubseteq \xi \Big\}
\end{equation}
\begin{equation}
\rcoupling{\subssd_{\xi}} := \bigoplus_{\zeta \in \Upsilon^C_\xi} \rsubssd_\zeta
\qquad \qquad
\rbasis_{\rcoupling{\subssd_\xi}} := \bigcup_{\zeta \in \Upsilon^C_\xi}
\rbasis_{\rsubssd_\zeta}
\end{equation}
We introduce an extended training set:
$\Xi \times \{1, \dots,
N_\rbasis+1\}$, $N_{\rbasis} := \dim(\rcoupling{\subssd_\xi})$.
Given a pair $(\mu, j) \in \Xi \times \{1, \dots, N_\rbasis+1\}$, we define the associated right hand
side as
\begin{equation}
  g_{\mu,j}(\phi) :=
  \begin{cases}
    - a_\mu(\psi_j, \phi) & \qquad \text{if } j \le N_\rbasis\\
    \dualpair{f_\mu}{\phi} & \qquad \text{if } j = N_\rbasis+1\,,
  \end{cases}
\end{equation}
where $\psi_j$ denotes the $j$-th basis function of $\rbasis_{\rcoupling{\subssd_\xi}}$.
We then construct the reduced cell space $\widetilde \subssd_\xi$
as the classical reduced basis space
(cf.\ Alg.\ \ref{algo:local_greedy}, LocalGreedy)
with respect to the following parameterized local problem:
Given a pair $(\mu, j) \in \Xi \times \{1, \dots, N_\rbasis+1\}$, find $u_{\mu,j} \in \subssd_\xi$ such that
\begin{equation}
a_\mu(u_{\mu,j},\phi) = g_{\mu,j}(\phi) \qquad \forall \phi \in \subssd_\xi.
\end{equation}
The corresponding reduced solutions are hence defined as:
Find $\wt u_{\mu,j} \in \rsubssd_\xi$ such that:
\begin{equation}
a_\mu(\wt u_{\mu,j},\phi) = g_{\mu,j}(\phi) \qquad \forall \phi \in \rsubssd_\xi
\end{equation}
Both problems have unique solutions due to the coercivity and continuity of $a_\mu$.
For the LocalGreedy (Algorithm \ref{algo:local_greedy})  we use the standard Reduced Basis residual error estimator, i.e.
\begin{equation}
\norm{u_{\mu,j} - \wt u_{\mu,j}}_{\subssd_\xi} \leq \Delta_{cell}(\wt u_{\mu,j}) := \frac{1}{\alpha_{LB}(\mu)} \norm{ R_{\mu,j}(\wt u_{\mu,j}) }_{{\subssd'_\xi}}\,,
\end{equation}
with the local residual
\begin{eqnarray}
R_{\mu, j}: V_\xi &\rightarrow& V'_\xi\\
\varphi &\mapsto& g_{\mu, j}(\cdot) - a_\mu(\varphi, \cdot)
\nonumber
\end{eqnarray}
and a lower bound for the coercivity constant $\alpha_{LB}$.
The idea of using a local greedy to generate a local space for all
possible boundary values can also be found in \cite{phdiapichino,FrancescaAntonietti2015}.
\begin{algorithm2e}[t]
\DontPrintSemicolon
\SetAlgoVlined
\SetKwFunction{LocalGreedy}{LocalGreedy}
\SetKwInOut{Input}{Input}
\SetKwInOut{Output}{Output}
\Fn{\LocalGreedy{$\xi, \varepsilon_\mathrm{greedy}$}}{
  \Input{space identifier $\xi$,\\ greedy tolerance $\varepsilon_\mathrm{greedy}$}
  \Output{reduced local subspace $\rsubssd_\xi$}
  $\rbasis_{\rsubssd_\xi} \leftarrow \emptyset$ \;
  \While{
    $\max\limits_{\substack{
        \mathllap{\mu} \in \mathrlap{\Xi} \\
        \mathllap{j} \in \mathrlap{\{1, \dots, N_\rbasis +1 \}}}
    }\ \Delta_{cell}(\wt{u}_{\mu,j}) > \varepsilon_\mathrm{greedy}$
  }%
  {
    $\hat \mu,\jhat \leftarrow \argmax\limits_{\substack{
        \mathllap{\mu} \in \mathrlap{\Xi} \\
        \mathllap{j} \in \mathrlap{\{1, \dots, N_\rbasis +1 \}}}}%
    \ \Delta_{cell}(\wt{u}_{\mu,j})$\;
    \Find{$u_{\hat \mu, \jhat} \in \subssd_\xi$ \st}{
      $
      \qquad a_{\hat \mu}(u_{\hat \mu, \jhat},\phi) = g_{\hat \mu, \jhat}(\phi) \qquad \forall \phi \in \subssd_\xi
      $}
    $u_{\hat \mu, \jhat} \leftarrow u_{\hat \mu, \jhat} - \sum\limits_{\mathclap{\phi\in\rbasis_{\rsubssd_\xi}}} {( \phi, u_{\hat \mu, \jhat})_V} \, \phi$\;
    $\rbasis_{\rsubssd_\xi} \leftarrow \rbasis_{\rsubssd_\xi} \cup \left\{\norm{u_{\hat \mu, \jhat}}_V^{-1}u_{\hat \mu, \jhat}\right\}$\;
  }
  \Return $\spanset(\rbasis_{\rsubssd_\xi})$\;
}
\caption{LocalGreedy to construct local cell spaces $\rsubssd_\xi, \ \xi \in \Upsilon_0$}
\label{algo:local_greedy}
\end{algorithm2e}


\section{A-Posteriori Error Estimator}
\label{sec:a_posteriori}
The model reduction error in the ArbiLoMod has to be controlled.
To this end, an a posteriori error estimator is used
which should have the following properties: 
\begin{enumerate}[topsep=0pt,itemsep=-1ex,partopsep=1ex,parsep=1ex]
\item It is robust and efficient.
\item It is online-offline decomposable.
\item It is parallelizable with little amount of communication.
\item After a localized geometry change, the offline computed
data in unaffected regions
can be reused.
\item
It can be used to steer adaptive enrichment of the reduced local subspaces.
\end{enumerate}
All these requirements are fulfilled by the estimator presented in
the following. We develop localized bounds for the standard RB error estimator,
\begin{equation}
\Delta(\widetilde u_\mu) := \frac{1}{\alpha_\mu} \norm{R_\mu(\widetilde u_\mu)}_{V_h'}
\end{equation}
where $R_\mu(\widetilde u_\mu) \in V_h^\prime$ is the global residual given as
$
	\dualpair{R_\mu(\widetilde u_\mu)}{\varphi} = \dualpair{f_\mu}{\varphi } - a_\mu(\widetilde{u}_\mu, \varphi)
$.
This error estimator is known to be robust and efficient (\cite[Proposition 4.4]{HesthavenRozzaEtAl2016}):
\begin{equation}\label{eq:global_estimator}
\norm{u_\mu - \widetilde u_\mu}_V \leq \Delta(\widetilde u_\mu) 
\leq \frac{\gamma_\mu}{\alpha_\mu} \norm{u_\mu - \widetilde u_\mu}_V.
\end{equation}
\subsection{Abstract Estimates}
We start by showing two abstract localized estimates for the dual norm
of a linear functional.
\begin{proposition}
\label{thm:a_posteriori}
Let $\{\subsod_\xi\}_{\xi \in \Upsilon_E}$ be
a collection of linear subspaces of $V_h$ for some finite index set $\Upsilon_E$ and let $\tilde V \subset V_h$ denote an arbitrary 
subspace. 
Moreover let $P_{\subsod_\xi}: V_h \longrightarrow \subsod_\xi \subseteq V_h, \xi \in \Upsilon_E$
be continuous linear mappings which satisfy $\sum_{\xi \in \Upsilon_E} P_{\subsod_\xi} = \operatorname{id}_{V_h}$.
With the stability constant of this partition modulo $\tilde V$ defined as
\begin{equation}
\puconstantV := \sup_{\varphi \in V_h \setminus \{0\}} \frac{( \sum_{\xi \in \Upsilon_E} \inf_{\tilde \varphi \in \tilde V \cap \subsod_\xi}  \norm{P_{\subsod_\xi}(\varphi) - \tilde \varphi}_V^2)^\frac{1}{2}}{\norm{\varphi}_V},
\nonumber
\end{equation}
we have for any linear functional $f \in V_h^\prime$ with $ \dualpair{f}{\tilde \varphi} = 0 \  \forall {\tilde \varphi} \in \tilde V$  the estimate
\begin{equation}
	\label{eq:abstract_localized_estimate}
	\norm{f}_{V_h^\prime} \leq \puconstantV \cdot \Big( \sum_{\xi \in \Upsilon_E} \norm{f}^2_{\subsod_\xi^\prime} \Big)^{\frac{1}{2}},
	\nonumber
\end{equation}
where $\norm{f}_{\subsod_\xi^\prime}$ denotes the norm of the restriction of $f$ to $\subsod_\xi$.
\end{proposition}

\begin{proof}
Using the Cauchy-Schwarz inequality and $\dualpair{f}{\tilde \varphi} = 0 \  \forall {\tilde \varphi} \in \tilde V$, we have
\begin{align*}
\norm{f}_{V_h'}
&= \sup_{\varphi \in V_h \setminus \{0\}} \frac{\sum_{\xi \in \Upsilon_E}  \dualpair{f}{P_{\subsod_\xi}(\varphi)}}{\norm{\varphi}_V}
= \sup_{\varphi \in V_h \setminus \{0\}} \frac{\sum_{\xi \in \Upsilon_E}  \inf_{\tilde \varphi \in \tilde V \cap \subsod_\xi}  \dualpair{f}{P_{\subsod_\xi}(\varphi)- \tilde \varphi}}{\norm{\varphi}_V }\\
&\leq \sup_{\varphi \in V_h \setminus \{0\}} \frac{\sum_{\xi \in \Upsilon_E} \norm{f}_{\subsod_\xi^\prime} \inf_{\tilde \varphi \in \tilde V \cap \subsod_\xi} \norm{P_{\subsod_\xi}(\varphi)- \tilde \varphi}_V}{\norm{\varphi}_V}\\
&\leq \sup_{\varphi \in V_h \setminus \{0\}} \frac{(\sum_{\xi \in \Upsilon_E} \norm{f}_{\subsod_\xi^\prime}^2)^\frac{1}{2}( \sum_{\xi \in \Upsilon_E} \inf_{\tilde \varphi \in \tilde V \cap \subsod_\xi}  \norm{P_{\subsod_\xi}(\varphi)- \tilde \varphi}_V^2)^\frac{1}{2}}{\norm{\varphi}_V}\\
&= \puconstantV \cdot  \Big( \sum_{\xi \in \Upsilon_E} \norm{f}^2_{\subsod_\xi^\prime} \Big)^{\frac{1}{2}}.
\end{align*}
\end{proof}

A stability constant very similar to $\puconstantV$ appears in the
analysis of overlapping domain decomposition methods
(e.g. \cite[Assumption 2.2]{toselli2005domain}, \cite{Spillane2013})
and in localization of error estimators on stars (e.g. \cite{CDN2012}).
\begin{proposition}
\label{thm:efficiency}
With the assumptions in Proposition \ref{thm:a_posteriori},
let $\dot{\bigcup}_{j=1}^J \Upsilon_{E,j} = \Upsilon_E$ be a
partition of $\Upsilon_E$ such that
\begin{equation*}
	\forall 1\leq j \leq J\ \forall \xi_1 \neq \xi_2 \in \Upsilon_{E,j}:\ \subsod_{\xi_1} \perp \subsod_{\xi_2}.
\end{equation*}
Then we have
\begin{equation}
	\label{eq:abstract_efficiency_estimate}
	\Big( \sum_{\xi \in \Upsilon_E} \norm{f}^2_{\subsod_\xi^\prime} \Big)^{\frac{1}{2}}
	\leq \sqrt{J}  \norm{f}_{V_h^\prime},
	\nonumber
\end{equation}
\end{proposition}
\begin{proof}
Let $\subsod_{\xi_1} \perp \subsod_{\xi_2}$ be some subspaces of $V_h$, and 
let $f$ be a continuous linear functional on 
$\subsod_{\xi_1} \oplus \subsod_{\xi_2}$.
If $v_{f,1} \in \subsod_{\xi_1}$ and $v_{f,2} \in \subsod_{\xi_2}$ are the 
Riesz representatives of the restrictions of $f$ to 
$\subsod_{\xi_1}$ and $\subsod_{\xi_2}$,
then due to the orthogonality of $\subsod_{\xi_1}$ and $\subsod_{\xi_2}$,
$v_{f,1} + v_{f,2}$ is the Riesz representative
of $f$ on $\subsod_{\xi_1} \oplus \subsod_{\xi_2}$. Thus,
\begin{align*}
\norm{f}_{(\subsod_{\xi_1} \oplus \subsod_{\xi_2})^\prime}^2 &= \norm{v_{f,1} + v_{f,2}}_{V_h}^2 \\
                                  &= \norm{v_{f,1}}_{V_h}^2 + \norm{v_{f,2}}_{V_h}^2 = \norm{f}_{\subsod_{\xi_1}^\prime}^2 + \norm{f}_{\subsod_{\xi_2}^\prime}^2,
\end{align*}
where we have used the orthogonality of the spaces again. The same is 
true for a larger orthogonal sum of spaces.
We therefore obtain:
\begin{align*}
\sum_{\xi \in \Upsilon_E} \norm{f}_{\subsod_\xi^\prime}^2 &= \sum_{j=1}^{J} \sum_{\xi \in \Upsilon_{E,j}} \norm{f}_{\subsod_\xi^\prime}^2 \\
    &= \sum_{j=1}^{J} \norm{f}_{(\bigoplus_{\xi \in \Upsilon_{E,j}}\subsod_\xi)^\prime}^2 \\
    &\leq J \norm{f}_{V_h^\prime}^2.
\end{align*}
\end{proof}

When grouping the spaces $\subsod_\xi$ so that in each group, all
spaces are orthogonal to each other, $J$ is the number of groups needed.
Applying both estimates to the residual, we obtain an efficient, localized error estimator:
\begin{corollary}\label{thm:abstract_error_estimate}
The error estimator $\Delta_{loc}(\widetilde{u}_\mu)$ defined as
\begin{equation}
\Delta_{loc}(\widetilde{u}_\mu) := \frac{1}{\alpha_\mu}  \puconstantV  \big(\sum_{\xi \in \Upsilon_E} \norm{R_\mu(\widetilde{u}_\mu)}^2_{\subsod_\xi^\prime} \big)^\frac{1}{2}
\end{equation}
is robust and efficient:
\begin{equation}
	\label{eq:efficiency_estimate}
	\norm{ u_\mu - \widetilde{u}_\mu }_V \leq \Delta_{loc}(\widetilde{u}_\mu)
	\leq \frac{\gamma_\mu  \sqrt{J}  \puconstantV}{\alpha_\mu}   \norm{ u_\mu - \widetilde{u}_\mu }_V
	\nonumber
\end{equation}
\end{corollary}
\begin{proof}
Applying Propositions \ref{thm:a_posteriori} and \ref{thm:efficiency} to
the error estimator
\\
$\Delta(\widetilde u_\mu) = \frac{1}{\alpha_\mu} \norm{R_\mu(\widetilde{u}_\mu)}_{V_h'}$ yields,
together with (\ref{eq:global_estimator}), the proposition.
\end{proof}

Online-offline decomposition of this error estimator can be done by 
applying the usual strategy for online-offline decomposition
used with the standard RB error estimator (see e.g. \cite[Sec.  4.2.5]{HesthavenRozzaEtAl2016}
or the numerically more stable approach \cite{BEOR14a})
to every dual norm in $\Delta_{loc}(\widetilde u_\mu)$.
\subsection{Choosing Spaces}
The error estimator defined in Corollary \ref{thm:abstract_error_estimate}
works for any spaces $\subsod_\xi$ and mappings $P_{\subsod_\xi}$
fulfilling the assumptions in Proposition \ref{thm:a_posteriori}.
However, in order to obtain good constants $\puconstantV$ and $\sqrt J$,
both have to be chosen carefully.
In addition, two more properties are needed
for good performance of the implementation.
First, the subspaces should be spanned by FE ansatz functions, allowing the
residual to be easily evaluated on these spaces. Second, the inner product
matrix on the subspaces should be sparse, as the inner product matrix
has to be solved in the computation of the dual norms.
We use an overlapping decomposition based on
the non-overlapping domain decomposition introduced in
\eqref{eq:dd}.
\begin{definition}[Overlapping space decomposition]
Let the index set $\Upsilon_E$ for the overlapping space decomposition be
given by the vertices of the domain decomposition, i.e.
\begin{equation}
\Upsilon_E = \Upsilon_2.
\end{equation}
We then define the overlapping spaces $\subsod_\xi$ supported on the
overlapping domains $\Omega_\xi$ by:
\begin{equation}
\subsod_\xi :=
\bigoplus \Big\{\subsbd_\zeta \ \Big| \ \zeta \subseteq \xi \Big\},
\qquad
\Omega_\xi := \overset{\circ}{\overline{\bigcup_{i \in \xi} \Omega_i}}
\qquad \xi \in \Upsilon_E.
\label{eq:overlapping_space_def}
\end{equation}
\end{definition}
Note that we have $\subsod_\xi = \{\psi \in \mathcal{B}\ |\ \overset{\circ}{\operatorname{supp}}(\psi) \subseteq \Omega_\xi\} \subseteq H^1_0(\Omega_\xi)$.
Contrary to $V_\xi$ or $U_\xi$, these spaces do not form a direct sum decomposition of $V_h$.
We next state a first estimate on the partition of unity constant of Corollary (\ref{thm:abstract_error_estimate}) for this choice 
of partition, which does not take into account that the residual vanishes on the reduced space. The resulting estimate thus 
depends on $H^{-1}$, where $H$ is the minimum diameter of the subdomains of the macro partition.
Typically the size of the macro partition is moderate such that $H^{-1}$ is small.
However, in the following Proposition \ref{thm:pu_bound} we will show that the constant $\puconstantV$ can be actually bounded 
independent of $H$, when we choose a partition of unity that is contained in the reduced space $\tilde V$.
\begin{proposition}
\label{thm:pu_bound_0}
Let $\ovlpconstant := \max_{x \in \Omega} \#\{\xi \in \Upsilon_E \ | \ x \in \Omega_\xi\}$ be the maximum 
number of estimator domains $\Omega_\xi$ overlapping in any point $x$ of $\Omega$
and let $H_\xi:= \diam (\Omega_\xi)$, $\xi \in \Upsilon_E$ and $H := \min_{\xi \in \Upsilon_E} H_\xi$.
Furthermore, assume that there exist partition of unity functions $p_\xi \in H^{1,\infty}(\Omega)$, $\xi \in \Upsilon_E$
and a linear interpolation operator $\mathcal{I}: V \longrightarrow V_h$ such that
\begin{enumerate}[topsep=0pt,itemsep=-1ex,partopsep=1ex,parsep=1ex,label=(\roman*)]
	\item $\sum_{\xi \in \Upsilon_E} p_\xi(x) = 1$ for all $x \in \Omega$,
	\item $\max_{\xi \in \Upsilon_E} \norm{p_\xi}_\infty \leq 1$ and $\norm{\nabla p_\xi}_\infty \leq
		\pufuncconstant  H_\xi^{-1}$ for all $\xi \in \Upsilon_E$,
	\item $\mathcal{I}(\varphi) = \varphi$ for all $\varphi \in V_h$,
	\item $\mathcal{I}(p_\xi V_h) \subseteq \subsod_\xi$ for all $\xi \in \Upsilon_E$,
	\item $\|\mathcal{I}(p_\xi v_h) - p_\xi v_h\|_V \leq c_I \|v_h\|_{\Omega_\xi,1}$ for all $\xi \in \Upsilon_E, v_h \in V_h$. 
\end{enumerate}
Then we have:
\begin{equation*}
	\puconstantV \leq \sqrt{4 + 2c_I^2 + 4\left(\pufuncconstant  H^{-1}\right)^2}\cdot \sqrt{\ovlpconstant}.
\end{equation*}
\end{proposition}
\begin{proof}
We compute the bound for $\puconstant$ using the partition of unity and the interpolation operator.  To this end, let
\begin{equation*}
	P_{\subsod_\xi}(\varphi) := \mathcal{I}(p_\xi  \varphi), \qquad \xi \in \Upsilon_E.	
\end{equation*}
Due to (iv), these are linear mappings $V_h \longrightarrow \subsod_\xi$, and using (i) and (iii) we obtain
$\sum_{\xi \in \Upsilon_E} P_{\subsod_\xi}(\varphi) = \mathcal{I}(\sum_{\xi \in \Upsilon_E} p_\xi  \varphi)
= \mathcal{I}(\varphi) = \varphi$ for all $\varphi \in V_h$.
Thus, Corollary \ref{thm:abstract_error_estimate} applies with this specific choice of partition operators $P_{\subsod_\xi}$.
Now, using (ii) and (v) we have for any $\varphi \in V_h$
\begin{align*}
   \sum_{\xi \in \Upsilon_E} \norm{P_{\subsod_\xi}(\varphi)}_V^2 
   &\leq 2 \sum_{\xi \in \Upsilon_E} \|\mathcal{I}(p_\xi\varphi) - p_\xi\varphi\|_V^2 + \norm{p_\xi  \varphi}_V^2 \\
   &\leq 2 \sum_{\xi \in \Upsilon_E} c_I^2 \|\varphi\|_{\Omega_\xi,1}^2 + \int_{\Omega_\xi} 
      2|\nabla p_\xi  \varphi|^2(x) + 2|p_\xi  \nabla \varphi|^2(x) + |p_\xi  \varphi|^2(x) dx \\
   &\leq 2 \sum_{\xi \in \Upsilon_E} c_I^2 \|\varphi\|_{\Omega_\xi,1}^2 +
     (1 + 2\left(\pufuncconstant  H^{-1}\right)^2) |\varphi|_{\Omega_\xi, 0}^2 + 2|\varphi|_{\Omega_\xi, 1}^2 \\
     &\leq (4 + 2c_I^2 + 4\left(\pufuncconstant  H^{-1}\right)^2) \ovlpconstant  \norm{\varphi}_V^2
\end{align*}
This gives us the estimate.
\end{proof}

\begin{remark}
	When the domain decomposition $\Omega_i$ is sufficiently regular (e.g. see the numerical examples below), partition of
	unity functions satisfying (i) and (ii) can easily be found. 
	If $V_h \cup \{p_\xi\,|\, \xi \in \Upsilon_E\}$ consists of $p$-th order finite element basis functions for some fine triangulation of $\Omega$,
	Lagrange interpolation can be chosen as interpolation operator $\mathcal{I}$. In fact, using standard interpolation error estimates
	and inverse inequalities one sees that for each element $T$ of the fine triangulation with diameter $h$ one has:
	\begin{align*}
		\|\mathcal{I}(p_\xi v_h) - p_\xi v_h\|_{T,1} & \leq c h^p |p_\xi v_h|_{T, p+1} \\
						   & \leq c^\prime h^p \sum_{k=1}^p |v_h|_{T, k} \cdot |p_\xi|_{T, p+1-k, \infty} \\
					           & \leq c^{\prime\prime} h^p \sum_{k=1}^p h^{-(k-1)} |v_h|_{T,1} \cdot h^{-(p+1-k)} |p_\xi|_{T,0,\infty} \\
					           & \leq c^{\prime\prime} p |v_h|_{T,1},
	\end{align*}
	where $c^{\prime\prime}$ is a constant bounded by the shape regularity of the fine triangulation.
\end{remark}
For the rectangular domain decomposition used in 
the numerical example below, the constant $J$ is $J = 2^d = 4$:
it is possible to divide the overlapping domains into four classes, so that
within each class, no domain overlaps with any other
(cf.\ \cite[Sec. 5]{Chung2015}).

Furthermore, 
the 
coercivity constant $\alpha_\mu$ and the stability constant $\gamma_\mu$,
or estimates, are required.
For the numerical example
presented in Section \ref{sec:numerical_example}, those
can be calculated analytically.
In general this is not possible and the details of estimating them
numerically are
subject for further investigations.
The numerical computation of a lower bound for the coercivity constant
was subject of extensive research in the RB community
(see e.g. \cite{HUYNH2007473,CHEN20081295})%
, but these methods require the calculation of the
coercivity constant at some parameter values and thus
require the solution of a global, fine scale problem.
To the authors' knowledge, there are no publications
on localization of these methods so far.

The upper bound on the constant $\puconstantV$ in Proposition \ref{thm:pu_bound_0} depends on the domain size $H$ approximately
like $H^{-1}$. As the domain size is considered a constant in the
context of ArbiLoMod, the error estimator is already considered efficient with this bound. In the next proposition, we however 
show that the constant can indeed be bounded independent of $H$, if we exploit that the residual vanishes on the 
reduced space $\tilde V$. 
\begin{proposition}
\label{thm:pu_bound}
Let $p_\xi$, $\xi \in \Upsilon_E$ be a partition of unity and $\mathcal{I}$
an interpolation operator satisfying the prerequisites of Proposition~\ref{thm:pu_bound_0}.
Furthermore, assume $V = H^1_0(\Omega)$ and that $p_\xi \in \tilde{V} \cap \subsod_\xi$ for $\xi \in
\Upsilon_E^{\rm int} := \{\xi \in \Upsilon_E | \ \bar \Omega_\xi \cap \partial \Omega = \emptyset\}$,
e.g. $p_\xi$ is chosen as a basis function of $\tilde{V}_\xi$ (see Subsections \ref{sec:decomposition} and \ref{sec:codim_2_spaces}
above).
Then the following estimate holds:
\begin{equation*}
\puconstantV \leq \sqrt{4 + 2c_I^2 + 4(\pufuncconstant c_{\mathrm pc})^2} \cdot \sqrt{\ovlpconstant},
\end{equation*}
with a Poincar$\rm \acute{e}$-inequality constant $c_{\mathrm pc}$ (see proof below) that does not depend 
on the fine or coarse mesh sizes ($h, H$). 
\end{proposition}
\begin{proof}
For arbitrary $\varphi \in V_h$ let $\bar \varphi_\xi := \frac{1}{|\Omega_\xi|}\int_{\Omega_\xi} \varphi$.
We then have with $\Upsilon_E^{\rm ext}:=
\Upsilon_E \setminus \Upsilon_E^{\rm int}$
\begin{eqnarray*}
\puconstantV &=& \sup_{\varphi \in V_h \setminus \{0\}} \frac{( \sum_{\xi \in \Upsilon_E} \inf_{\tilde \varphi \in \tilde V \cap \subsod_\xi}  \norm{P_{\subsod_\xi}(\varphi) - \tilde \varphi}_V^2)^\frac{1}{2}}{\norm{\varphi}_V} \\
	&\leq& \sup_{\varphi \in V_h \setminus \{0\}}  \!\!\! \frac{( 
	      \sum_{\xi \in \Upsilon_E^{\rm int}}  \norm{P_{\subsod_\xi}(\varphi) - \bar \varphi_\xi p_\xi}_V^2
	      + \sum_{\xi \in \Upsilon_E^{\rm ext}}  \norm{P_{\subsod_\xi}(\varphi)}_V^2
	      )^\frac{1}{2}}{\norm{\varphi}_V},
\end{eqnarray*}
where we have used that by construction $\bar \varphi_\xi p_\xi \in { \tilde V \cap \subsod_\xi}$ for all $\xi \in
\Upsilon_E^{\rm int}$.
For any $\varphi \in V_h$ and $\xi \in \Upsilon_E^{\rm int}$ we then have 
$\norm{P_{\subsod_\xi}(\varphi) - \bar \varphi_\xi p_\xi}_V^2 \leq 2c_I^2 \|\varphi\|_{\Omega_\xi,1}^2 + 2
    \norm{(\varphi - \bar \varphi_\xi) p_\xi}_V^2 $, where
\begin{align*}
    \norm{(\varphi - \bar \varphi_\xi) p_\xi}_V^2 
    &\leq \int_{\Omega_\xi} 
      2|\nabla (\varphi - \bar \varphi_\xi) p_\xi |^2(x) + 2| (\varphi - \bar \varphi_\xi) \nabla p_\xi |^2(x) dx \\
      &\hspace*{5em} + \|(\varphi - \bar \varphi_\xi) p_\xi\|^2_{L^2(\Omega_\xi)}.
 \end{align*}
With a rescaled Poincar$\rm \acute{e}$-type inequality 
$$
	\norm{\varphi - \bar \varphi_\xi}_{L^2(\Omega_\xi)} \leq c_{\mathrm pc} H_\xi \norm{\nabla \varphi}_{L^2(\Omega_\xi)},
$$
and
$
	\norm{\varphi - \bar \varphi_\xi}_{L^2(\Omega_\xi)} \leq \norm{\varphi}_{L^2(\Omega_\xi)},
$
we get
\begin{align*}
   \int_{\Omega_\xi} 
   2|\nabla (\varphi - \bar{\varphi}_\xi) p_\xi |^2(x) &+ 2| (\varphi - \bar \varphi_\xi) \nabla p_\xi |^2(x) dx + \|(\varphi - \bar \varphi_\xi) p_\xi\|^2_{L^2(\Omega_\xi)}\\
       & \qquad\qquad\leq (2 + 2(\pufuncconstant c_{\mathrm pc})^2 )  \norm{\nabla \varphi}^2_{L^2(\Omega_\xi)} +
   \|\varphi\|^2_{L^2(\Omega_\xi)}\\
 \end{align*}
In analogy we obtain for the boundary terms, i.e. $\xi \in \Upsilon_E^{\rm ext}$, the estimates
$\norm{P_{\subsod_\xi}(\varphi)}_V^2 \leq 2c_I^2 \|\varphi\|_{\Omega_\xi,1} + 2
    \norm{\varphi p_\xi}_V^2 $, and
\begin{align*}
    \norm{ \varphi p_\xi }_V^2 
    &\leq \int_{\Omega_\xi} 
    2|\nabla \varphi p_\xi|^2(x) + 2| \varphi \nabla p_\xi |^2(x) dx + \| \varphi  p_\xi \|^2_{L^2(\Omega_\xi)} \\
    & \leq (2 + 2(\pufuncconstant c_{\mathrm pc})^2) \norm{\nabla \varphi}^2_{L^2(\Omega_\xi)} +
      \|\varphi\|^2_{L^2(\Omega_\xi)}.
 \end{align*}
using a rescaled Poincar$\rm \acute{e}$-type inequality which holds for $\xi \in \Upsilon_2^{\rm ext}$ as 
$\varphi \in V_h$ has zero boundary values, i.e.
$$
	\norm{\varphi}_{L^2(\Omega_\xi)} \leq c_{\mathrm pc} H_\xi \norm{\nabla \varphi}_{L^2(\Omega_\xi)}.
$$
Summing up all contributions we then have

\begin{align*}
  \sum_{\xi \in \Upsilon_E^{\rm int}} & \norm{P_{\subsod_\xi}(\varphi) - \bar \varphi_\xi p_\xi}_V^2
    + \sum_{\xi \in \Upsilon_E^{\rm ext}}  \norm{P_{\subsod_\xi}(\varphi)}_V^2 \\
	      &\leq \sum_{\xi \in \Upsilon} 2c_I^2\|\varphi\|_{\Omega_\xi,1}^2 + 2\left[ (2 + 2(\pufuncconstant c_{\mathrm pc})^2) \norm{\nabla
      \varphi}^2_{L^2(\Omega_\xi)} + \|\varphi\|^2_{L^2(\Omega_\xi)} \right]\\
	    & \leq (4 + 2c_I^2 + 4(\pufuncconstant c_{\mathrm pc})^2) \ovlpconstant  \norm{\varphi}^2_{V}.
\end{align*}	      
This gives us the estimate.
\end{proof}


Proposition \ref{thm:pu_bound} gives a bound on $\puconstantV$ that depends on the contrast of the underlying 
diffusion coefficient if $p_\xi \in \tilde V$, $\xi \in \Upsilon_E^{\rm int}$ is chosen as the MsFEM type hat functions 
as suggested in Section \ref{sec:decomposition} above. However, it is independent on the mesh sizes $h, H$. 
A crucial ingredient to obtain this bound is the fact that we 
included this macroscopic partition of unity in our reduced approximation space $\tilde V$. 
If alternatively we would chose $p_\xi \in \tilde V$ to be the traditional Lagrange hat functions, 
the bound on $\puconstantV$ in  Proposition \ref{thm:pu_bound} would be independent of the contrast.
In fact, we might expect that $\puconstantV$ behaves much better then the upper bound due to 
the approximation properties of the reduced space. It would actually be possible to compute $\puconstantV$
for given $V_h, \tilde V$ which would however be computationally expensive and thus not of any use 
in pratical applications. Proposition \ref{thm:pu_bound}, however shows that the localized a posteriori 
error estimator in Corollary \ref{thm:abstract_error_estimate} in the context of ArbiLoMod is indeed robust and efficient, 
even with respect to $H \to 0$.


Comparing with other localized RB and multiscale methods,
one observes a difference in the scaling of the efficiency constants.
While in our case, $c_{pu}$ is independent of both $h$ and $H$,
the a posteriori error estimator published for LRBMS has a
$H/h$ dependency
\cite[Theorem 4.6]{OS15}
and in the certification framework for SCRBE, a $h^{-1/2}$ scaling appears
\cite[Proposition 4.5]{Smetana2015a}.
The error estimators published for GMsFEM in \cite{Chung2014b} also 
have no dependency on $H$ or $h$. However, they also rely on specific
properties of the basis generation. 
Also in the analysis of the ``Discontinuous Galerkin Reduced Basis Element Method'' (DGRBE),
Pacciarini et.al.~have
a factor of $h^{-1/2}$ in the
a priori analysis \cite{FrancescaAntonietti2015}
and in the a posteriori error estimator \cite{PhdPacciarini}.
\subsection{Local Efficiency}
So far we did not use properties of the bilinear form other than
coercivity and continuity. Assuming locality of the bilinear form 
as in \eqref{eq:heat_equation}, we get a local efficiency
estimate and an improved global efficiency estimate.
\begin{proposition}
\label{thm:local_efficiency}
Let the bilinear form $a$ be given by \eqref{eq:heat_equation}.
Then we have the localized efficiency estimate
\begin{equation}
\norm{R_\mu(\widetilde{u}_\mu)}_{\subsod_\xi'} \leq
\gamma_\mu |u_\mu - \widetilde{u}_\mu|_{\Omega_\xi, 1}.
\end{equation}
\end{proposition}
\begin{proof}
Using the error identity 
\begin{equation}
a_\mu(u_\mu - \widetilde{u}_\mu, \varphi) = \dualpair{R_\mu(\widetilde{u}_\mu) }{ \varphi },
\nonumber
\end{equation}
we obtain for any $\varphi \in \subsod_\xi$
\begin{align*}
	\dualpair{R_\mu(\widetilde{u}_\mu)}{\varphi} &= \int_\Omega \sigma_\mu(x)\nabla(u_\mu - \widetilde{u}_\mu)(x) \nabla \varphi(x) dx \\
	               &= \int_{\Omega_\xi} \sigma_\mu(x)\nabla(u_\mu - \widetilde{u}_\mu)(x) \nabla \varphi(x) dx \\
		       &\leq \gamma_\mu |u_\mu - \widetilde{u}_\mu|_{\Omega_\xi,1} \norm{\varphi}_V,
\end{align*}
from which the statement follows.
\end{proof}
\begin{remark}
Under the assumptions of Proposition \ref{thm:local_efficiency},
it is easy to see that we have the improved efficiency estimate
\begin{equation}
\Delta_{loc}(\widetilde{u}_\mu)
	\leq \frac{\gamma_\mu  \sqrt{\ovlpconstant}  \puconstantV}{\alpha_\mu}   \norm{ u_\mu - \widetilde{u}_\mu }_V.
	\nonumber
\end{equation}
In many cases, a better constant can be found. Finite Element ansatz functions
are usually not orthogonal if they share support. So if 
$\ovlpconstant$ spaces have support in one point in space, they have
to be placed in different groups when designing a partition
for Proposition \ref{thm:efficiency}, so $\ovlpconstant \leq J$
 (cf.\ \cite[p. 67]{toselli2005domain}).
\end{remark}
\subsection{Relative Error Bounds}
From the error estimators for the absolute error, we can construct 
error estimators for the relative error. Estimates for the relative error
are given in \cite[Proposition 4.4]{HesthavenRozzaEtAl2016},
but the estimates used here are slightly sharper.
\begin{proposition}
Assuming $\norm{\widetilde u_\mu}_V > \Delta(\widetilde u_\mu)$
and  $\norm{\widetilde u_\mu}_V > \Delta_{loc}(\widetilde u_\mu)$,
the error estimators defined by
\begin{align*}
\Delta^{rel}(\widetilde u_\mu) 
    &:= \frac{\Delta(\widetilde u_\mu)}{\norm{\widetilde u_\mu}_V - \Delta(\widetilde u_\mu)}
\\
\Delta^{rel}_{loc}(\widetilde u_\mu) 
    &:= \frac{\Delta_{loc}(\widetilde u_\mu)}{\norm{\widetilde u_\mu}_V - \Delta_{loc}(\widetilde u_\mu)}
\end{align*}
are robust and efficient:
\begin{eqnarray*}
\frac{\norm{u_\mu - \widetilde u_\mu}_V}{\norm{u_\mu}_V}
&\leq \Delta^{rel}(\widetilde u_\mu)
&\leq \left( 1 + 2 \Delta^{rel} (\widetilde u_\mu)\right) \frac{\gamma_\mu}{\alpha_\mu}
\frac{\norm{u_\mu - \widetilde u_\mu}_V}{\norm{u_\mu}_V}
\\
\frac{\norm{u_\mu - \widetilde u_\mu}_V}{\norm{u_\mu}_V}
&\leq \Delta^{rel}_{loc}(\widetilde u_\mu)
&\leq \left( 1 + 2 \Delta^{rel}_{loc}(\widetilde u_\mu) \right) \frac{\gamma_\mu \sqrt{J} \puconstantV}{\alpha_\mu}
\frac{\norm{u_\mu - \widetilde u_\mu}_V}{\norm{u_\mu}_V}
\end{eqnarray*}
\end{proposition}
\begin{proof}
Realizing that $\left(\norm{\widetilde u_\mu}_V - \Delta(\widetilde u_\mu) \right)\leq \norm{u_\mu}_V$,
it is easy to see that
\begin{equation}
\frac{\norm{u_\mu - \widetilde u_\mu}_V}{\norm{u_\mu}_V}
\leq
\frac{\Delta(\widetilde u_\mu)}{\norm{u_\mu}_V}
\leq
\frac{\Delta(\widetilde u_\mu)}{\norm{\widetilde u_\mu}_V - \Delta(\widetilde u_\mu)},
\end{equation}
which is the first inequality. 
Using 
$$
\norm{\widetilde u_\mu}_V + \Delta(\widetilde u_\mu) = \left(\norm{\widetilde u_\mu}_V - \Delta(\widetilde u_\mu) \right) \left( 1 + 2 \Delta^{rel}(\widetilde u_\mu) \right)
$$
the second inequality can be shown:
\begin{align*}
\Delta^{rel}(\widetilde u_\mu) = 
\frac{\Delta(\widetilde u_\mu)}{\norm{\widetilde u_\mu}_V - \Delta(\widetilde u_\mu)}
& \leq \frac{\gamma_\mu}{\alpha_\mu} \frac{\norm{u_\mu - \widetilde u_\mu}_V}{\norm{\widetilde u_\mu}_V - \Delta(\widetilde u_\mu)}\\
&=  \frac{\gamma_\mu}{\alpha_\mu} \frac{\norm{u_\mu - \widetilde u_\mu}_V}{\norm{\widetilde u_\mu}_V + \Delta(\widetilde u_\mu)} \left( 1 + 2 \Delta^{rel}(\widetilde u_\mu) \right)\\
&\leq \frac{\gamma_\mu}{\alpha_\mu} \frac{\norm{u_\mu - \widetilde u_\mu}_V}{\norm{u_\mu}_V} \left( 1 + 2 \Delta^{rel}(\widetilde u_\mu) \right).
\end{align*}
The inequalities for $\Delta^{rel}_{loc}$ can be shown accordingly.
\end{proof}

Reviewing the five desired properties of an a posteriori error estimator
at the beginning of this section, we see that the presented
error estimator is robust and efficient (1) and is online-offline
decomposable (2). Parallelization can be done over the spaces
$\subsod_\xi$. Only online data has to be transferred, so there is 
little communication (3). The online-offline decomposition
only has to be repeated for a space $\subsod_\xi$, if a new basis
function with support in $\Omega_\xi$ was added. So reuse in unchanged
regions is possible (4). How the adaptive enrichment is steered (5)
will be described in the following section.

\section{Enrichment Procedure}
\label{sec:enrichment}
The first ArbiLoMod solution is obtained using the initial reduced local subspaces generated
using the local training and greedy procedures described in Section
\ref{sec:training_and_greedy}.
If this solution is not good enough according
to the a posteriori error estimator, the solution is improved by
enriching the reduced local subspaces and then solving the global reduced
problem again.
The full procedure is given in Algorithm \ref{algo:online_enrichment}
and described in the following.

For the enrichment, we use the overlapping local subspaces introduced
in \eqref{eq:overlapping_space_def}, which are also used for the
a posteriori error estimator.
Local problems are solved in the overlapping spaces.
The original bilinear form is used, but as a right hand side
the residual of the last reduced solution is employed.
The local spaces and the parameter values for which the enrichment
is performed are selected in a D\"orfler-like \cite{dorfler1996convergent}
algorithm.
The thus obtained local solutions $u_l$ do not fit into our
space decomposition, as they lie in one of the overlapping spaces, not
in one of the local subspaces used for the basis construction. Therefore,
the $u_l$ are decomposed using the projection operators $\subsmap_{\subssd_\xi}$
defined in Definition~\ref{def:local_projection_operators}.
In the setting of our numerical example (Section~\ref{sec:numerical_example}), this decomposition yields at most 9 parts
(one codim-2 part, four codim-1 parts and four codim-0 parts).
Of these parts, the one worst approximated by the existing reduced
local subspace is selected for enrichment. ``Worst approximated'' is here defined
as having the largest part orthogonal to the existing reduced local subspace.
We denote the part of $\subsmap_{\subssd_\xi}(u_l)$
orthogonal to $\rsubssd_\xi$ w.r.t.\ the inner product of $V$
by $(\subsmap_{\subssd_\xi}(u_l))^\perp$.

To avoid communication, cell spaces $\rsubssd_\xi, \xi \in \Upsilon_0$
are not enriched at this point. Such an enrichment would require the
communication of the added basis vector, which might be large. Instead,
only the other spaces are enriched, and the cell spaces associated with $\Upsilon_0$
are regenerated using the greedy procedure from Section \ref{sec:codim_0_greedy}.
For the other spaces, a strong compression of the basis vectors is possible
(cf.\ Section \ref{sec:runtime_and_communication}).

This selection of the local spaces can lead to one reduced local space
being enriched several times in one iteration. Numerical experiments have
shown that this leads to poorly conditioned systems, as
the enrichment might try to introduce the same feature into a local basis twice.
To prevent this, the enrichment algorithm enriches each reduced
local subspace at most once per
iteration.
\begin{algorithm2e}
\DontPrintSemicolon
\SetAlgoVlined
\SetKwFunction{OnlineEnrichment}{OnlineEnrichment}
\SetKwInOut{Input}{Input}
\SetKwInOut{Output}{Output}
\Fn{\OnlineEnrichment{$d, \mathrm{tol}$}}{
  \Input{enrichment fraction $d$,\\ target error $\mathrm{tol}$}
  \While{
    $\max\limits_{\mu \in \Xi}
    \ \Delta(\wt{u}_\mu) > \mathrm{tol}$
  }%
  {
    $E \leftarrow \emptyset$\;
    \While{
      $\left(\sum\limits_{(\mu,\xi) \in E} \norm{R_\mu(\widetilde{u}_\mu)}_{(\subsod_\xi)'} \right)
      / \left(\sum\limits_{(\mu,\xi) \in (\Xi \times \Upsilon_E)} \norm{R_\mu(\widetilde{u}_\mu)}_{(\subsod_\xi)'} \right)< d$
    }%
    {
    $\hat \mu,\hat \xi \leftarrow
    \argmax\limits_
    {(\mu, \xi) \in \left(\Xi \times \Upsilon_E\right) \setminus E}
    \norm{R_\mu(\widetilde{u}_\mu)}_{(\subsod_\xi)'}$\;
    $E \leftarrow E \cup (\hat \mu, \hat \xi)$\;
    }
    \tcc{$S$ is used for double enrichment protection.}
    $S \leftarrow \emptyset$\;
    \For{$(\mu,\xi) \in E$}
    {
    \Find{$u_l \in \subsod_{\xi}$ \st}{
      $a_{\mu}(u_l, \varphi) =
      \dualpair{R_\mu(\widetilde{u}_{\mu})}{\varphi} \qquad \forall
      \varphi \in \subsod_{\xi}$}
    $\check\xi \leftarrow \argmax\limits_{\xi \in \Upsilon \setminus \Upsilon_0}\norm{ (\subsmap_{\subssd_\xi}(u_l))^\perp } _{\subssd_\xi}$\;
    \If{$\check \xi \notin S$}{
      $\rsubssd_{\check\xi} \leftarrow \rsubssd_{\check\xi} \oplus \spanset((\subsmap_{\subssd_{\check\xi}}(u_l))^\perp)$\;
      $S \leftarrow S \cup \check\xi$\;
    }
    }
    run \texttt{LocalGreedys}\;
    recalculate reduced solutions\;
  }
}
\caption{Online Enrichment}
\label{algo:online_enrichment}
\end{algorithm2e}


\section{Handling Local Changes}
\label{sec:onchange}
When the enrichment iteration has converged to a sufficient accuracy,
the obtained solution is handed over to the user. The envisioned implementation
should then wait for the user to modify the model under consideration.
After a change to the model simulated, the full procedure described
above is repeated, but wherever possible, existing data is reused.
On the changed domains, a new mesh is generated, if necessary. 
Basis vectors having support in the changed region are discarded.
The domain decomposition is never changed and thus has to be 
independent of the geometry.
The potential savings are not only in the reduced basis generation,
but also in the assembly of the system matrices. In an implementation
featuring localized meshing and assembly domain by domain, 
meshing and assembly has to be repeated only in the domains
affected by the change.
This approach of recomputing everything in the changed
region was chosen because it allows a robust implementation
without any assumptions about the changes. 

\section{Runtime and Communication}
\label{sec:runtime_and_communication}
A major design goal of ArbiLoMod is communication avoidance and
scalability in parallel environments. Although the main topic of this
publication are the mathematical properties, we
want to highlight the possibilities offered by ArbiLoMod to
reduce communication in a parallel setup.

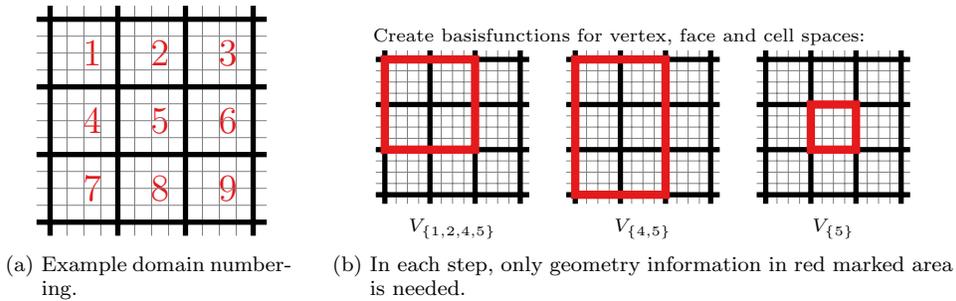
\begin{figure}
\footnotesize
\centering
\begin{subfigure}[b]{0.3\textwidth}
\centering
\begin{minipage}{0.8\linewidth}
  \begin{tikzpicture}[scale=0.9]
    \node at (0.5+0.125,2.5) [color=Xred] {\huge 1};
    \node at (1.5+0.125,2.5) [color=Xred] {\huge 2};
    \node at (2.5+0.125,2.5) [color=Xred] {\huge 3};
    \node at (0.5+0.125,1.5) [color=Xred] {\huge 4};
    \node at (1.5+0.125,1.5) [color=Xred] {\huge 5};
    \node at (2.5+0.125,1.5) [color=Xred] {\huge 6};
    \node at (0.5+0.125,0.5) [color=Xred] {\huge 7};
    \node at (1.5+0.125,0.5) [color=Xred] {\huge 8};
    \node at (2.5+0.125,0.5) [color=Xred] {\huge 9};
    \draw [step=2.5mm,color=gray,very thin] (-0.2,-0.2) grid (3.2,3.2);
    \draw [step=1cm,black,line width=0.7mm] (-0.2,-0.2) grid (3.2,3.2);
  \end{tikzpicture}
\end{minipage}
\caption{Example domain numbering.}
\end{subfigure}\hfill
\begin{subfigure}[b]{0.65\textwidth}
\centering
\def\myscale{0.6}
\begin{tabular}{ccc}
\multicolumn{3}{l}{Create basisfunctions for vertex, face and cell spaces:}\\
\begin{tikzpicture}[scale=\myscale]
\draw [step=2.5mm,color=gray,very thin] (-0.2,-0.2) grid (3.2,3.2);
\draw [step=1cm,black,line width=0.7mm] (-0.2,-0.2) grid (3.2,3.2);
\draw [color=Xred,line width=1.1mm] (0,1) rectangle (2,3);
\end{tikzpicture}&
\begin{tikzpicture}[scale=\myscale]
\draw [step=2.5mm,color=gray,very thin] (-0.2,-0.2) grid (3.2,3.2);
\draw [step=1cm,black,line width=0.7mm] (-0.2,-0.2) grid (3.2,3.2);
\draw [color=Xred,line width=1.1mm] (0,0) rectangle (2,3);
\end{tikzpicture}&
\begin{tikzpicture}[scale=\myscale]
\draw [step=2.5mm,color=gray,very thin] (-0.2,-0.2) grid (3.2,3.2);
\draw [step=1cm,black,line width=0.7mm] (-0.2,-0.2) grid (3.2,3.2);
\draw [color=Xred,line width=1.1mm] (1,1) rectangle (2,2);
\end{tikzpicture}\\[.5ex]
$V_{\{1,2,4,5\}}$&
$V_{\{4,5\}}$&
$V_{\{5\}}$\\
\end{tabular}
\caption{In each step, only geometry information in red marked area is needed.}
\end{subfigure}
\vspace*{-2ex}
\caption{
Before online enrichment, it is possible to compute all reduced basis
function having support on a domain ($\Omega_5$ here) using only
local information about the domain and its surrounding domains.
}
\label{fig:no_communication}
\end{figure}

Similar to overlapping Domain Decomposition methods we require, that
not only the local domain, but also an overlap region is available
locally. For a subdomain $\Omega_i$ the overlap region is the domain
itself and all adjacent domains, as depicted in Figure~\ref{fig:no_communication},
i.e.\ all subdomains in the neighborhood $\neigh_{\{i\}}$.
As the overlap region includes the
support of all training spaces, one can compute all initial reduced local subspaces
with support in $\Omega_i$ without further communication.
This work can be distributed on many nodes.
Afterwards, only reduced representations of the operator have to be communicated.
Using the operator decomposition
$a^b(u,v) = \sum_{i=1}^{N_D} a_{\Omega_i}^b(u,v)$, a global, reduced
operator is collected using an all-to-one communication of reduced matrices.
The global reduced problem is then solved on a single node.
It is assumed that the global, reduced system is sufficiently small.

If the accuracy is not sufficient, online enrichment is
performed. This step requires additional communication; first for the
evaluation of the error estimator and second to communicate new basis
vectors of reduced face spaces $\widetilde V_\xi, \ \xi \in
\Upsilon_1$.
Note that it is sufficient to communicate the local projection
$\subsmap_{\subsbd_\xi}(\psi)$ and reconstruct the actual basis
function as its extension, so that we save communication costs
proportional to the volume to surface ratio.

The evaluation of the localized error estimator requires the dual
norms of the residual in the localized spaces
$\norm{R_\mu(\widetilde u_\mu)}_{\subsod_\xi'}$. Using a stabilized online-offline
splitting \cite{BEOR14a}, the evaluation of the error estimator can be
evaluated for the full system using only reduced quantities. The
computation of the reduced operators is performed in parallel, similar
to the basis construction in the first step. The actual evaluation of
the error estimator can be performed on a single node and is
independent of the number of degrees of freedom of the high fidelity
model.

\label{sec:H_explanation}
An important parameter for ArbiLoMod's runtime is the domain size $H$.
The domain size affects the size of the local problems, 
the amount of parallelism in the algorithm,
and the size of the reduced global problem.
An $H$ too large leads to large local problems, while an $H$ too small
leads to a large reduced global problem (see also the numerical
example in Section \ref{sec:num_results} and especially
the results in Table \ref{tab:H_study}).
$H$ has to be chosen
to balance these two effects. As the focus of this manuscript 
is ArbiLoMod's mathematical properties, 
the question of choosing $H$ for optimal performance 
will be postponed to future research.


\section{Numerical Example}\label{sec:numerical_example}
The numerical experiments were performed using \mbox{pyMOR} \cite{pymor}.
The source code for the reproduction of all results presented in this
section are provided as a supplement to this paper. See the
README file therein for installation instructions.
Note that this code is kept simple to easily explore ArbiLoMod's
mathematical properties. It is not tuned for performance.
First results for electrodynamics were 
published in \cite{BEOR16}.
\label{sec:numerical_experiments}
\begin{figure}
\centering
\includegraphics[width=0.35\textwidth]{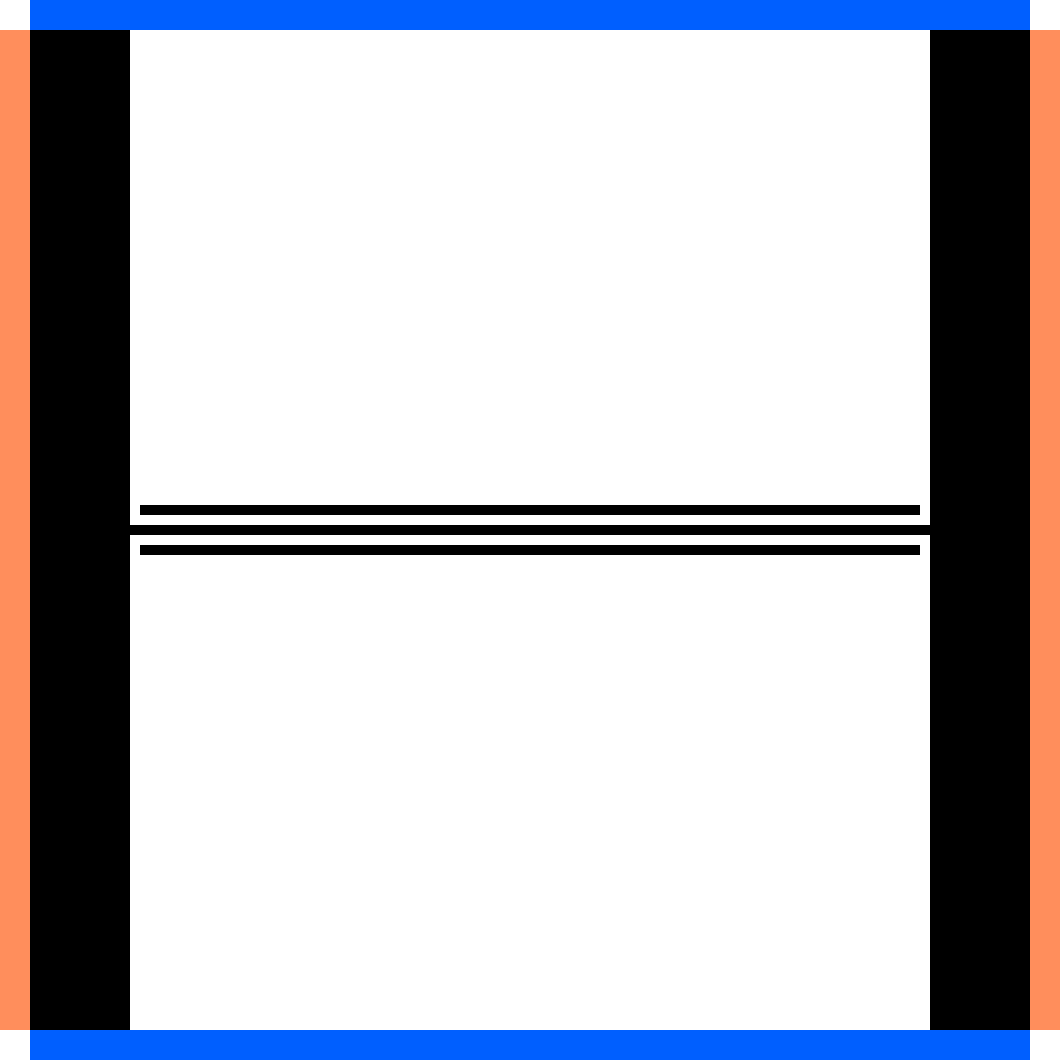}
\caption{First structure in sequence of simulated structures. Unit square with
high and low conductivity regions. White: constant conductivity region
($\sigma = 1$), black: parameterized high conductivity region
($\sigma_\mu = 1 + \mu$).
Homogeneous Neumann boundaries $\Gamma_N$ at top and bottom marked blue,
inhomogeneous Dirichlet boundaries $\Gamma_D$ at left and
right marked light red.}
\label{fig:geo1}
\end{figure}
\subsection{Problem Definition}
To illustrate the capabilities of ArbiLoMod we apply it to a sequence of
locally modified geometries. We consider
heat conduction without heat sources in the domain
on the unit square $\Omega := ]0,1[^2$.
We approximate $u$ solving
\begin{equation}
-\nabla \cdot ( \sigma_\mu \nabla u_\mu) = 0
\end{equation}
where $\sigma_\mu : \Omega \rightarrow \mathbb{R}$ is the heat conductivity.
We apply homogeneous Neumann boundaries at the top and the bottom:
$\nabla u \cdot n = 0 \ \mathrm{on} \ \Gamma_N
	:= ( ]0,1[ \ \times \ 0 ) \cup ( ]0,1[ \ \times \ 1 )
	$
and inhomogeneous Dirichlet boundaries at the left and right:
$
u = 1 \ \mathrm{on} \ \Gamma_{D,1} := 0 \ \times \ ]0,1[
$,
$
u = -1 \ \mathrm{on} \ \Gamma_{D,-1} := 1 \ \times \ ]0,1[
$
(see also Figure~\ref{fig:geo1}).
\begin{figure}
\footnotesize
\centering
\begin{tabular}{c|c|c|c|cc}
\# & left (zoom) & full structure & right (zoom) &
\multicolumn{2}{l}{solution for $\sigma_\mu = 10^5$}
\\
\hline
1 &
\encloseimage{\includegraphics[width=0.13\textwidth,frame,trim=50 450 850 450,clip]{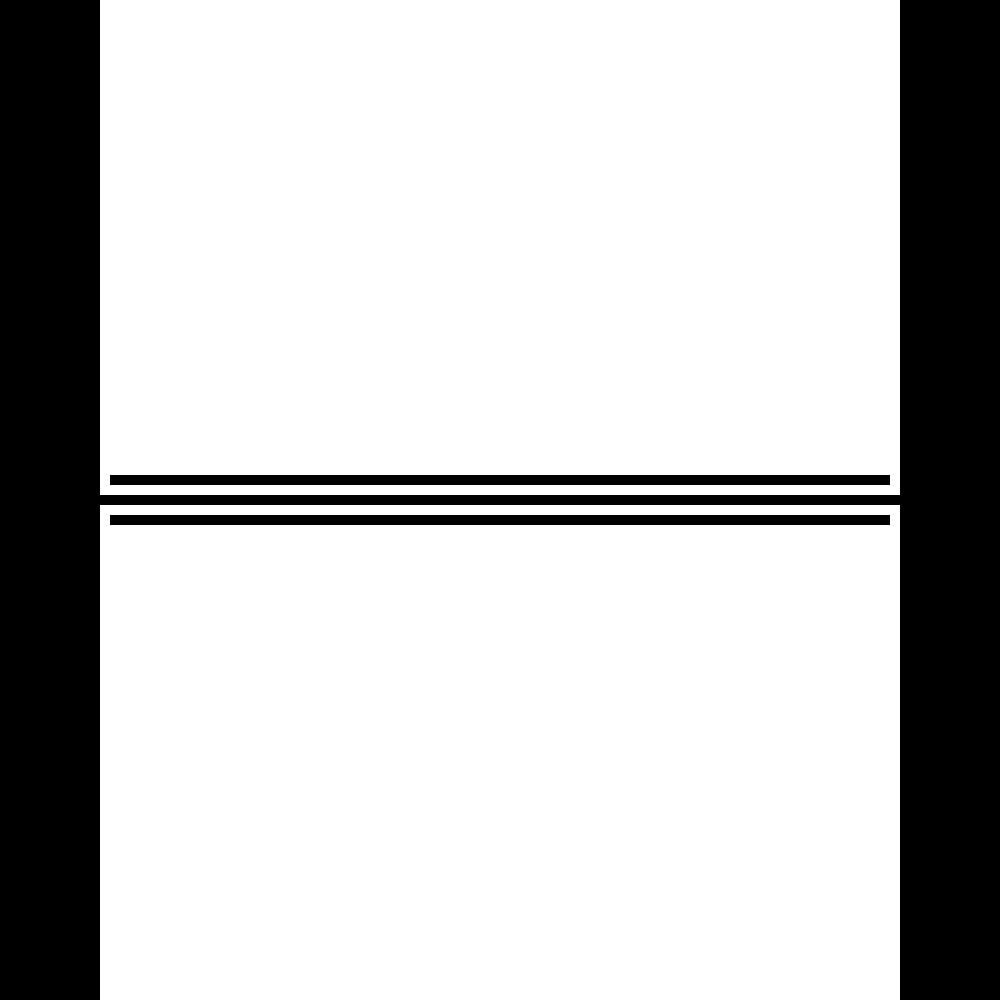}} &
\encloseimage{\includegraphics[width=0.13\textwidth,frame]{h_seq1.png}} &
\encloseimage{\includegraphics[width=0.13\textwidth,frame,trim=850 450 50 450,clip]{h_seq1.png}} &
\encloseimage{\includegraphics[width=0.13\textwidth]{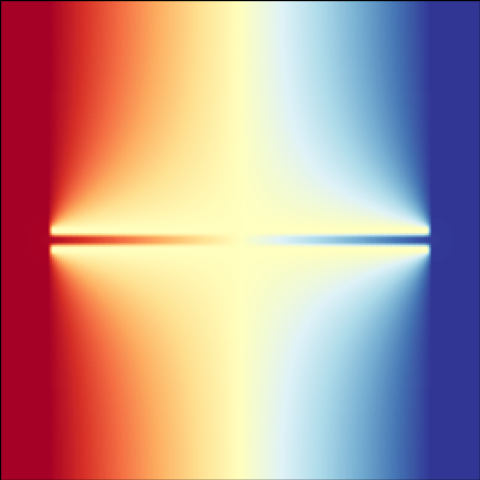}} &
\hspace{-25pt}
\encloseimage{
\includegraphics[width=0.05\textwidth]{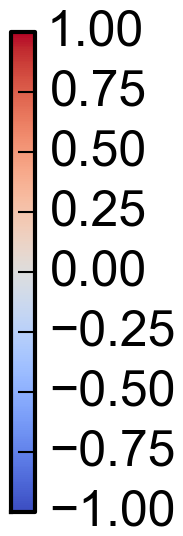}
}
\\
\hline
2 &
\encloseimage{\includegraphics[width=0.13\textwidth,frame,trim=50 450 850 450,clip]{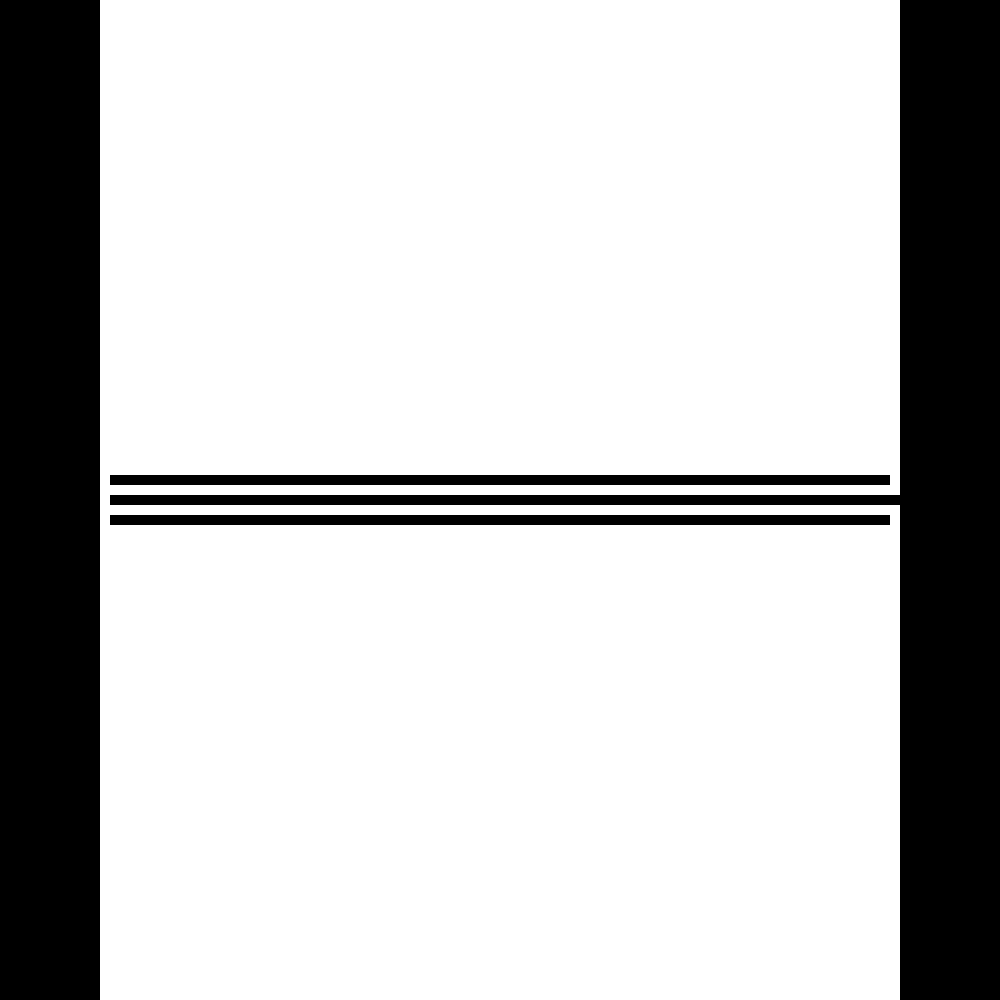}} &
\encloseimage{\includegraphics[width=0.13\textwidth,frame]{h_seq2.png}} &
\encloseimage{\includegraphics[width=0.13\textwidth,frame,trim=850 450 50 450,clip]{h_seq2.png}} &
\encloseimage{\includegraphics[width=0.13\textwidth]{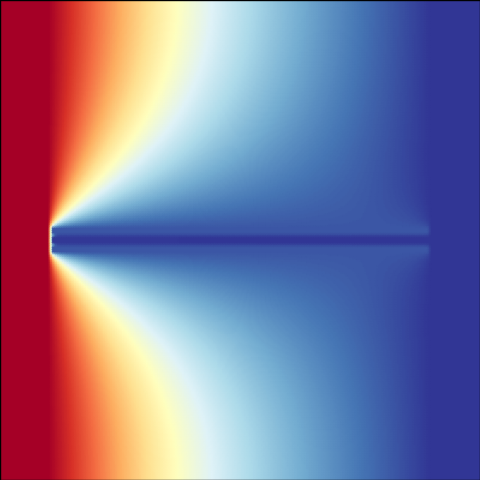}}
\\
\hline
3 &
\encloseimage{\includegraphics[width=0.13\textwidth,frame,trim=50 450 850 450,clip]{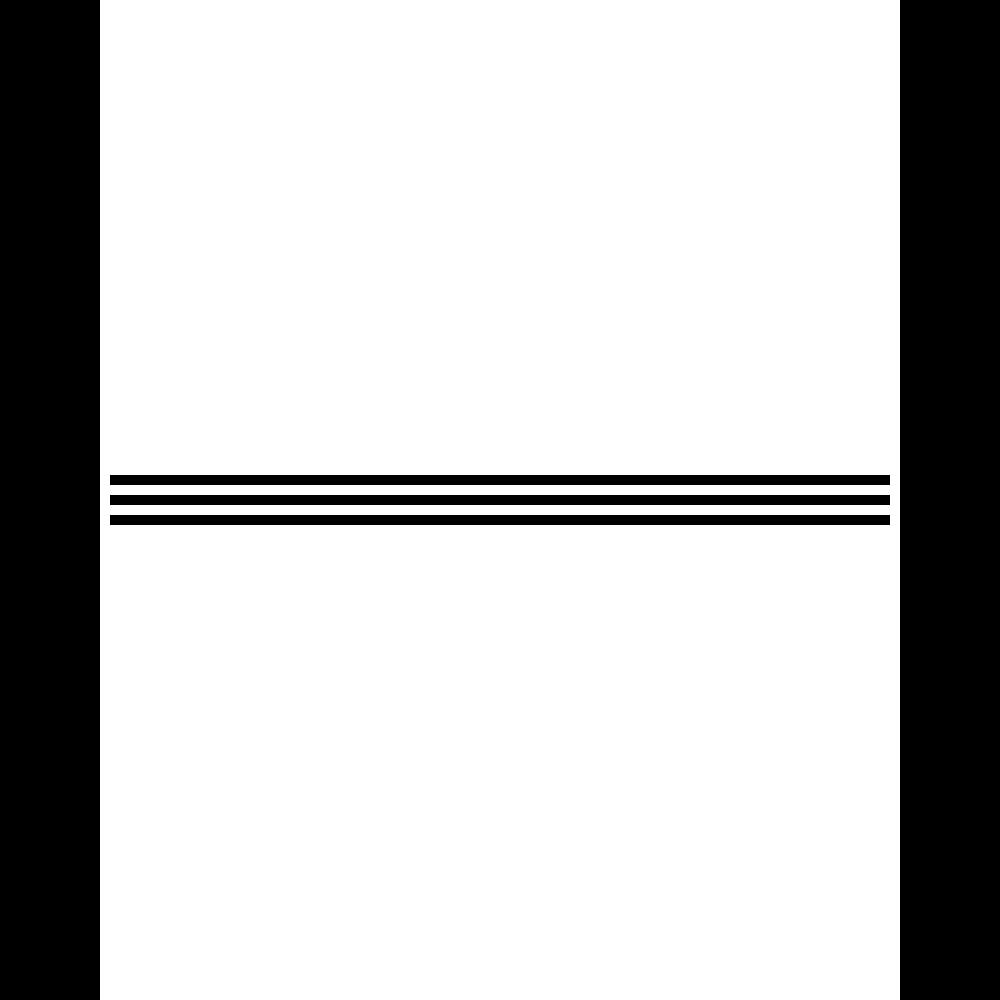}} &
\encloseimage{\includegraphics[width=0.13\textwidth,frame]{h_seq3.png}} &
\encloseimage{\includegraphics[width=0.13\textwidth,frame,trim=850 450 50 450,clip]{h_seq3.png}} &
\encloseimage{\includegraphics[width=0.13\textwidth]{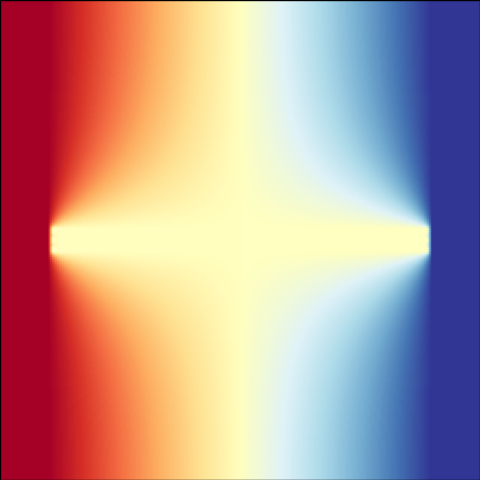}}
\\
\hline
4 &
\encloseimage{\includegraphics[width=0.13\textwidth,frame,trim=50 450 850 450,clip]{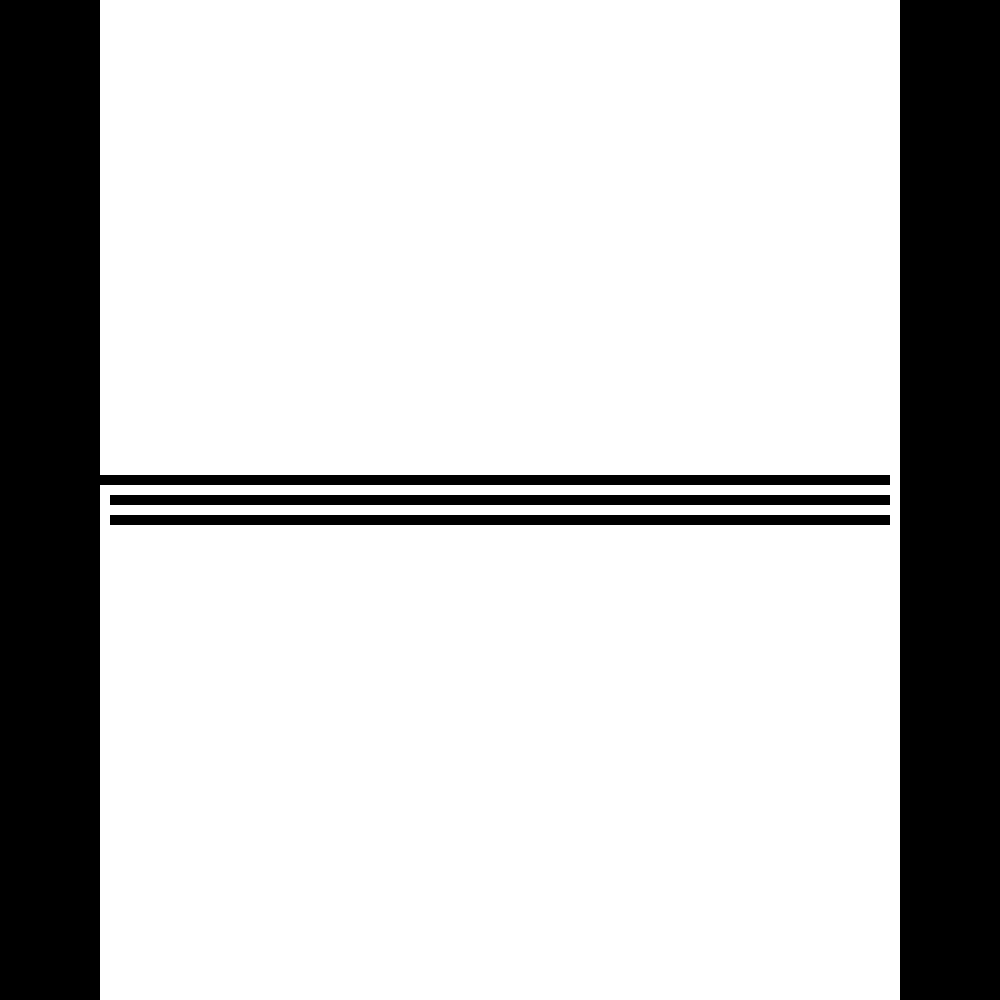}} &
\encloseimage{\includegraphics[width=0.13\textwidth,frame]{h_seq4.png}} &
\encloseimage{\includegraphics[width=0.13\textwidth,frame,trim=850 450 50 450,clip]{h_seq4.png}} &
\encloseimage{\includegraphics[width=0.13\textwidth]{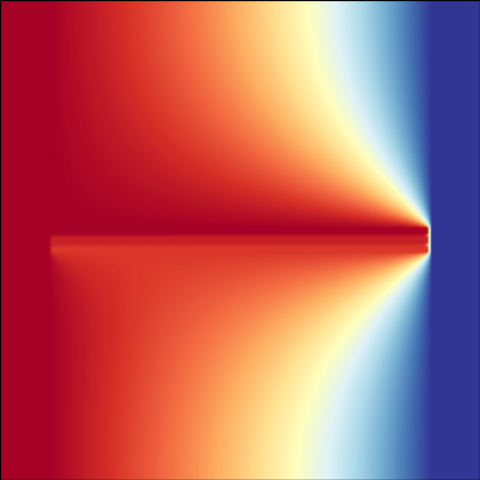}}
\\
\hline
5 &
\encloseimage{\includegraphics[width=0.13\textwidth,frame,trim=50 450 850 450,clip]{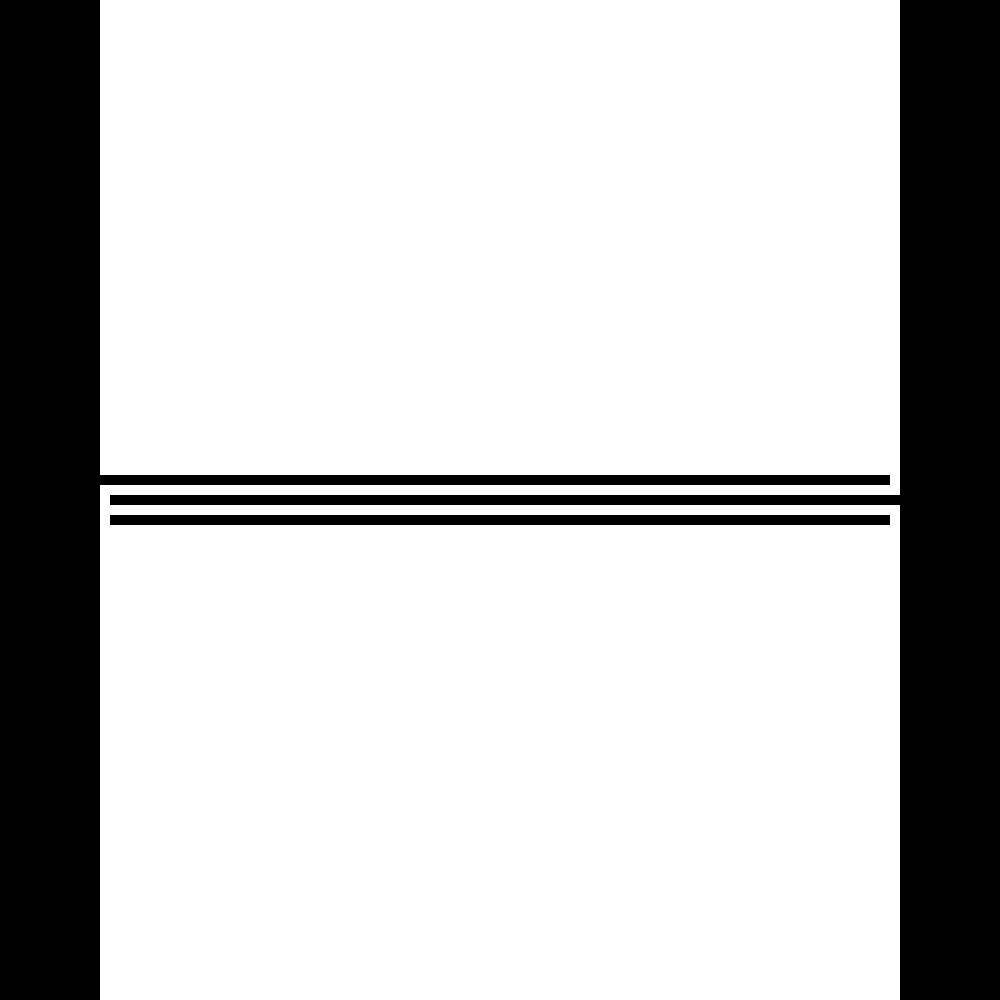}} &
\encloseimage{\includegraphics[width=0.13\textwidth,frame]{h_seq5.png}} &
\encloseimage{\includegraphics[width=0.13\textwidth,frame,trim=850 450 50 450,clip]{h_seq5.png}} &
\encloseimage{\includegraphics[width=0.13\textwidth]{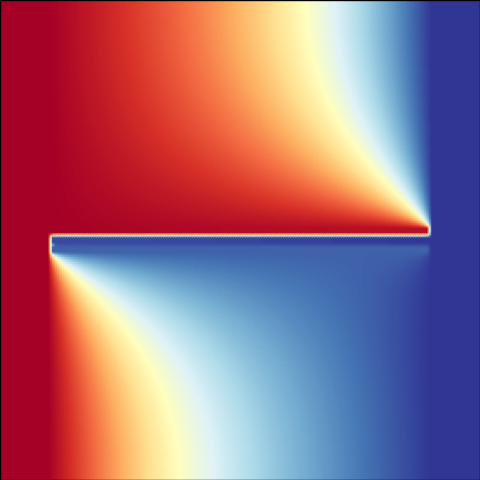}}
\end{tabular}
\vspace{10pt}
\caption{Sequence of structures simulated along with solutions for
one parameter value. Very localized changes cause strong global changes
in the solution. Script: \texttt{create\_full\_solutions.py}.}
\label{fig:structures}
\end{figure}


The unit square is partitioned into two regions:
one region with constant heat conductivity $\sigma_\mu = 1$ and
one region with constant, but parameterized
conductivity $\sigma_\mu = 1 + \mu$
where $\mu \in [10^0, 10^5]$.
We call this second region the ``high conductivity region'' $\Omega_{h,i}$.
For reproduction or benchmarking, we give
precise definitions in the following.
In the first geometry, the high conductivity region is
\begin{eqnarray}
\Omega_{h,1} &:=&
\big[(0.0 ,0.1) \times (0, 1) \big]
\ \cup \
\big[(0.9, 1.0) \times (0, 1) \big]
\ \cup \ \\ &&
\nonumber
\big[(0.11, 0.89) \times (0.475, 0.485) \big]
\ \cup \
\big[(0.1, 0.9) \times (0.495, 0.505) \big]
\ \cup \ \\ &&
\nonumber
\big[(0.11, 0.89) \times (0.515, 0.525) \big],
\end{eqnarray}
see also Figure~\ref{fig:geo1}.
The inhomogeneous Dirichlet boundary conditions are handled by a shift function
$u_s$ and we solve for $u_0$ having homogeneous Dirichlet values where
$u = u_0 + u_s$.
The parametrization of the conductivity in the high conductivity region,
$\sigma_\mu$, leads to a term in the affine decomposition of the bilinear form.
The affine decomposition of the bilinear form and linear form are
\begin{eqnarray}
a_\mu(u_0,\varphi) &=& \mu \int_{\Omega_{h,i}}  \nabla u_0 \cdot \nabla \varphi \dx + \int_{\Omega}  \nabla u_0 \cdot \nabla \varphi \dx \\
\nonumber
\dualpair{f_\mu}{\varphi} &=& - \mu \int_{\Omega_{h,i}} \nabla u_s \cdot \nabla \varphi \dx - \int_{\Omega} \nabla u_s \cdot \nabla \varphi \dx .
\end{eqnarray}
The coercivity constant $\alpha_\mu$ of the corresponding bilinear form with
respect to the $H^1$ norm is bounded from
below by $\alpha_{LB} := \frac{\sigma_\mathrm{min}}{c_F^2+1}$ where $\sigma_\mathrm{min}$
is the minimal conductivity and $c_F = \frac{1}{\sqrt{2}\pi}$ is the constant in the Friedrich's inequality
$
	\norm{\varphi}_{L^2(\Omega)} \leq c_F \norm{\nabla \varphi}_{L^2(\Omega)} \forall \varphi \in H^1_0(\Omega)
$.
The problem is discretized using $P^1$ ansatz functions on a structured triangle grid with maximum triangle size $h$. The grid is carefully constructed to resolve the
high conductivity regions, i.e. $h$ is chosen to be $1/n$ where $n$ is a multiple of 200.
To mimic ``arbitrary local modifications'', the high conductivity region is changed slightly four times, which leads to a sequence of five structures to be simulated in total. The high conductivity regions are defined as:
\begin{eqnarray}
\Omega_{h,2} &:=& \Omega_{h,1} \setminus \big[(0.1, 0.11) \times (0.495, 0.505) \big]\\
\Omega_{h,3} &:=& \Omega_{h,2} \setminus \big[(0.89, 0.9) \times (0.495, 0.505) \big] \nonumber\\
\Omega_{h,4} &:=& \Omega_{h,3} \cup \big[(0.1, 0.11) \times (0.515, 0.525) \big] \nonumber\\
\Omega_{h,5} &:=& \Omega_{h,4} \cup \big[(0.89, 0.9) \times (0.495, 0.505) \big]
\nonumber
\end{eqnarray}
These modifications only affect a very small portion of the domain
(actually, only 0.01\%), but for high contrast configurations, they lead to strong global
changes in the solution, see Figure~\ref{fig:structures}.
An equidistant domain decomposition of $8 \times 8$ domains is used. The mesh resolves the domain boundaries.
\subsection*{Configuration}
If not specified otherwise, we use
a mesh size of $1/h = 200$,
a training tolerance of $\varepsilon_\mathrm{train} = 10^{-4}$,
a number of random samplings of $M = 60$,
a greedy tolerance of $\varepsilon_\mathrm{greedy} = 10^{-3}$,
a convergence criterion of $\norm{R_\mu(\widetilde u_\mu)}_{V_h'} < 10^{-2}$,
an enrichment fraction of $d = 0.5$,
a training set of size $|\Xi| = 6$,
and the parameter for extension calculation is $\overline \mu = 10^5$.
\subsection{Results}
\label{sec:num_results}
\begin{figure}
\centering
\includegraphics[width=0.6\textwidth]{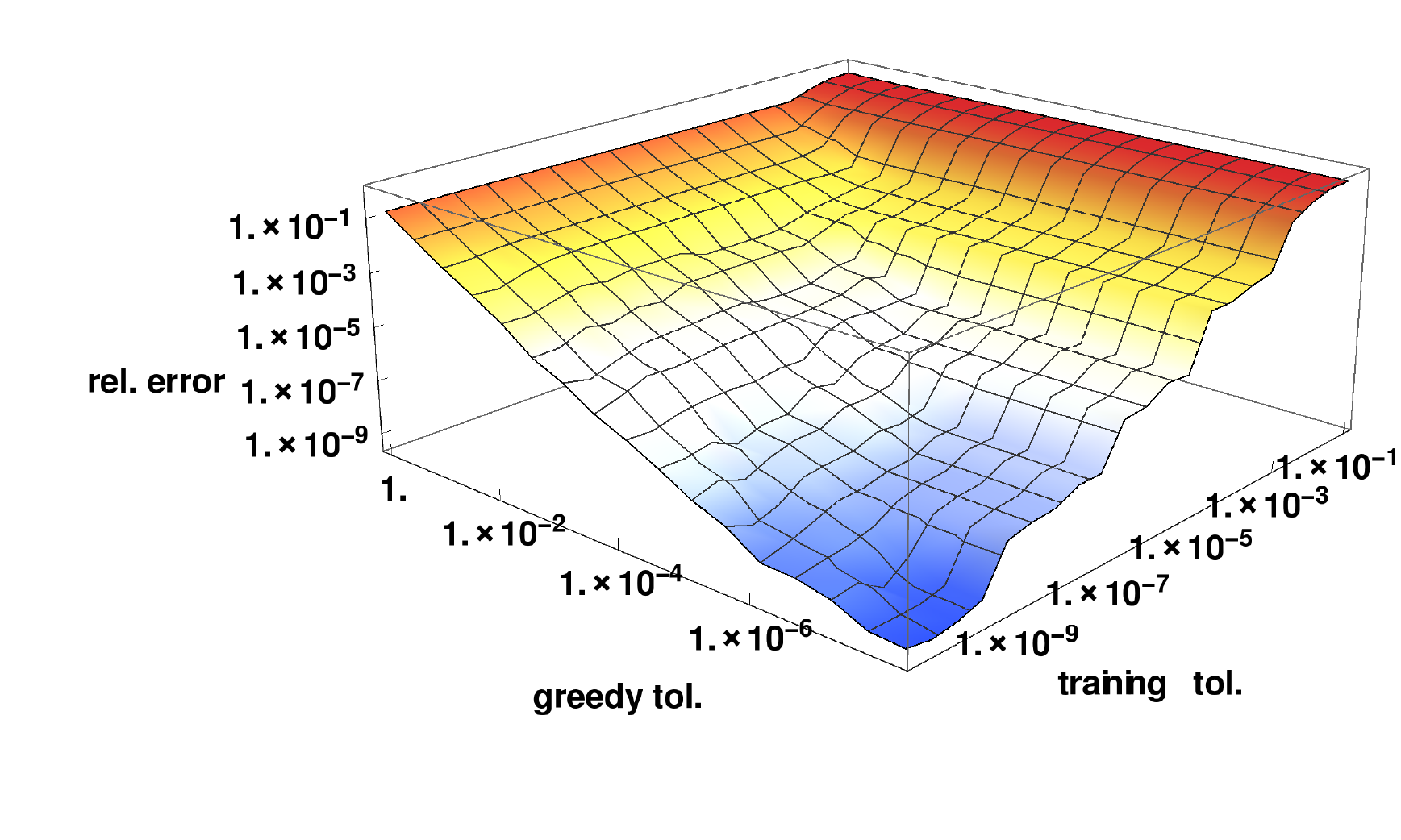}
\caption{Maximum relative $H^1$-error on training set $\Xi$ in dependence of tolerances in codim 1 training
and codim 0 greedy. Online enrichment disabled. Script:
\texttt{experiment\_tolerances.py}.}
\label{fig:tolerances}
\end{figure}
The initial reduced space is created using the local trainings and
greedy algorithms. In both the trainings and the greedy algorithms
a tolerance parameter steers the quality of the obtained reduced space:
In the trainings,
$\varepsilon_\mathrm{train}$
is the stopping criterion for the SnapshotGreedy
(Algorithm \ref{algo:snapshot_greedy}) and the local greedys
stop when the local error estimator stays below the prescribed tolerance
$\varepsilon_\mathrm{greedy}$.
The resulting reduction errors in
dependence on the two tolerances are depicted in Figure~\ref{fig:tolerances}.
\begin{table}
\footnotesize
\centering
\begin{tabu}{c|[1pt]c|[lightgray]c|c|[lightgray]c|c|[lightgray]c|[1pt]c|[lightgray]c}
\multirow{2}{*}{geometry}&
\multicolumn{6}{c|[1pt]}{\em with training} &
\multicolumn{2}{c}{\em without training}\\
&\multicolumn{2}{c|}{trainings} &
\multicolumn{2}{c|}{greedys} &
\multicolumn{2}{c|[1pt]}{iterations} &
\multicolumn{2}{c}{iterations}
\\
\hline
&\multicolumn{2}{c|}{reuse:} &
\multicolumn{2}{c|}{reuse:} &
\multicolumn{2}{c|[1pt]}{reuse:} &
\multicolumn{2}{c}{reuse:}
\\
 & no & yes & no & yes & no & yes & no & yes\\
\tabucline[lightgray]{2-}
1 & 112 & 112 (-0 \%)  & 64 & 64 (-0 \%)  & 24 & 24 (-0 \%)
& 46 & 46 (-0 \%) \\
2 & 112 & 5 (-96 \%)  & 64 & 8 (-88 \%)  & 24 & 13 (-46 \%)
& 48 & 28 (-42 \%) \\
3 & 112 & 5 (-96 \%)  & 64 & 8 (-88 \%)  & 20 & 14 (-30 \%)
& 42 & 27 (-36 \%) \\
4 & 112 & 3 (-97 \%)  & 64 & 6 (-91 \%)  & 25 & 10 (-60 \%)
& 54 & 23 (-57 \%) \\
5 & 112 & 5 (-96 \%)  & 64 & 8 (-88 \%)  & 25 & 12 (-52 \%)
& 52 & 27 (-48 \%) \\
\end{tabu}
\caption{Number of iterations of online enrichment:
(a) With and without codim 1 training.
(b) With and without reuse of basis
functions of previous simulations.
Convergence criterion: $\norm{R_\mu(\widetilde u_\mu)}_{V_h'} < 10^{-4}$,
greedy tolerance: $\varepsilon_\mathrm{greedy} = 10^{-5}$.
See also Figures \ref{fig:basisreusewithtraining}, \ref{fig:basisreuse}.\\
Scripts: \texttt{experiment\_basisreuse\_with\_training.py}, \texttt{experiment\_basisreuse.py}.}
\label{tab:basisreusecompare}
\end{table}

\begin{figure}
\centering
\input{figure_basisreuse_with_training}
\caption{Relative error over iteration with and without basis reuse
after geometry change. With codim 1 training.
Convergence criterion: $\norm{R_\mu(\widetilde u_\mu)}_{V_h'} < 10^{-4}$,
greedy tolerance: $\varepsilon_\mathrm{greedy} = 10^{-5}$.
See also Table~\ref{tab:basisreusecompare}.\\
Script: \texttt{experiment\_basisreuse\_with\_training.py}.}
\label{fig:basisreusewithtraining}
\end{figure}

\begin{figure}
\centering
\input{figure_basisreuse}
\caption{Relative $H^1$ error over iterations with and without basis reuse
after geometry change. Without codim 1 training.
Convergence criterion: $\norm{R_\mu(\widetilde u_\mu)}_{V_h'} < 10^{-4}$,
greedy tolerance: $\varepsilon_\mathrm{greedy} = 10^{-5}$.
See also Table~\ref{tab:basisreusecompare}.\\
Script: \texttt{experiment\_basisreuse.py}.}
\label{fig:basisreuse}
\end{figure}

\begin{figure}
\centering
\input{figure_estimatorperformance}
\caption{
Error estimators $\Delta^{rel}$ and
$\Delta^{rel}_{loc}$ over iterations, compared to the
relative error. Plotted for $\alpha_\mu = \puconstantV = 1$.
Simulation performed with
a convergence criterion of $\norm{R_\mu(\widetilde u_\mu)}_{V_h'} < 10^{-6}$,
a training tolerance of $\varepsilon_\mathrm{train} = 10^{-5}$,
and a greedy tolerance of $\varepsilon_\mathrm{greedy} = 10^{-7}$.
Script: \texttt{experiment\_estimatorperformance.py}.}
\label{fig:estimatorperformance}
\end{figure}

If the resulting error is too big, it can be further reduced using
iterations of online enrichment as depicted in Figure
\ref{fig:basisreusewithtraining}. Results suggest an
very rapid decay of the error with online enrichment.
The benefits of ArbiLoMod can be seen in Figure
\ref{fig:basisreusewithtraining} and Table
\ref{tab:basisreusecompare}: after
the localized geometry changes, most of the
work required in the initial basis creation does not need to be
repeated and the online enrichments converge faster for
subsequent simulations, leading to less iterations.
The online enrichment presented here
converges even when started on empty bases, as depicted in
Figure~\ref{fig:basisreuse}.
It does not rely on properties
of the reduced local subspaces created by trainings and greedys.
\newcommand{\rot}[1]{\multicolumn{1}{l}{\adjustbox{angle=24,lap=\width-1em}{#1}}‌​}
\begin{table}
\footnotesize
\begin{tabular}{r|r|r|r|r|r|r|r|r|r}
\multicolumn{1}{c}{$1/h$} & \rot{global dofs} & \rot{global solve time [s]} &
\rot{\# dofs, codim 1 training space} &
\rot{avg. time per codim 1 training [s]} &
\rot{\# dofs, codim 0 space} &
\rot{max time per codim 0 greedy [s]} & \rot{\# dofs, reduced problem} & \rot{solve time, reduced [ms]} & \rot{max error [\permil]}
\\
\hline
200 & 80,401 & 0.656 & 7,626 & 1.02 & 1,201 & 4.9 & 1,178 & 21.8 & 1.316\\
400 & 320,801 & 4.87 & 30,251 & 5.14 & 4,901 & 7.04 & 1,151 & 22.4 & 1.433\\
600 & 721,201 & 23.6 & 67,876 & 14 & 11,101 & 10.7 & 1,116 & 19.1 & 2.035\\
800 & 1,281,601 & 41.8 & 120,501 & 29.5 & 19,801 & 17.8 & 1,101 & 17.1 & 2.735\\
1000 & 2,002,001 & 86.4 & 188,126 & 51.3 & 31,001 & 24.2 & 1,089 & 18.8 & 1.351\\
1200 & 2,882,401 & 230 & 270,751 & 81.2 & 44,701 & 36.6 & 1,082 & 18.6 & 4.462\\
1400 & 3,922,801 & 230 & 368,376 & 120 & 60,901 & 51.7 & 1,073 & 18.2 & 2.379
\end{tabular}
\vspace{10pt}
\caption{Runtimes for selected parts of ArbiLoMod without
online enrichment.
``max error'' denotes $\max_{\mu \in \Xi} {\|u_\mu -\widetilde u_\mu\|_V}/{\|u_\mu\|_V}$.
Runtimes measured using
a pure Python implementation, using SciPy solvers (SuperLU sequential).
Note that the global solve time is for one parameter value
while training and greedy produce spaces valid for all parameter
values in the training set $\Xi$.
Script: \texttt{experiment\_create\_timings.py}.}
\label{tab:runtimes}
\end{table}
%

\begin{table}
\begin{tabular}{r|r|r|r|r|r|r|r|r|r}
\multicolumn{1}{c}{$1/H$} &
\rot{\# dofs, codim 1 training space} &
\rot{mean training time [s]} &
\rot{max training time [s]} &
\rot{\# dofs, codim 0 space} &
\rot{mean greedy time [s]} &
\rot{max greedy time [s]} &
\rot{\# dofs, reduced problem} &
\rot{solve time, reduced [ms]} &
\rot{max error [\permil]}
\\
\hline
4 & 30,251 & 10.9 & 14.1 & 4,901 & 3.07 & 6.6 & 403 & 3.85 & 0.362\\
5 & 19,401 & 6.88 & 8.74 & 3,121 & 2.32 & 11.7 & 517 & 4.38 & 0.435\\
8 & 7,626 & 2.63 & 3.55 & 1,201 & 1.35 & 5.69 & 1,178 & 14.9 & 1.32\\
10 & 4,901 & 1.68 & 1.94 & 761 & 0.655 & 3.98 & 1,451 & 19.8 & 0.220\\
20 & 1,251 & 0.483 & 0.661 & 181 & 0.305 & 3.18 & 5,025 & 94.4 & 0.0804
\end{tabular}
\caption{
Influence of domain size $H$. Fine mesh resolution: $1/h = 200$.
``max error'' is $\max_{\mu \in \Xi} {\|u_\mu -\widetilde u_\mu\|_V}/{\|u_\mu\|_V}$.
Smaller domains lead to more parallelism and smaller local problems, but
also to more global dofs and a worse constant in the a-posteriori error estimator.
Measured using
a pure Python implementation, using SciPy solvers (SuperLU sequential).
Script: \texttt{experiment\_H\_study.py}.
}
\label{tab:H_study}
\end{table}
%

The performance of the localized a posteriori error estimator
$\Delta^{rel}_{loc}$ can be seen in Figure~\ref{fig:estimatorperformance}.
Comparison of the localized estimator 
$\Delta^{rel}_{loc}$
with the global estimator
$\Delta^{rel}$
shows that, for the example considered here,
the localization does not add a significant factor beyond the factor
$\puconstantV$, which is supposed to be close to one and was thus 
neglected in Fig.\ \ref{fig:estimatorperformance}.

Even though our implementation was not tuned for performance and is not
parallel, we present some timing measurements in Table~\ref{tab:runtimes} and Table~\ref{tab:H_study}.
Table~\ref{tab:runtimes} shows that, already in our unoptimized implementation, trainings and greedys
have a shorter runtime than a single global solve for problems of sufficient 
size. Taking into account that trainings and greedys create a solution
space valid in the whole parameter space, this data is a strong hint
that ArbiLoMod can realize its potential for acceleration 
for large 
problems, in an optimized implementation and a parallel computing environment.
Especially when the solution is to be calculated at multiple parameter values.
Table~\ref{tab:H_study} shows the effect of choosing the domain size $H$: 
Large domains lead to large local problems, while small domains lead to a large
reduced global problem (see also Section \ref{sec:H_explanation}).
\begin{figure}
\centering
\includegraphics[width=0.35\textwidth]{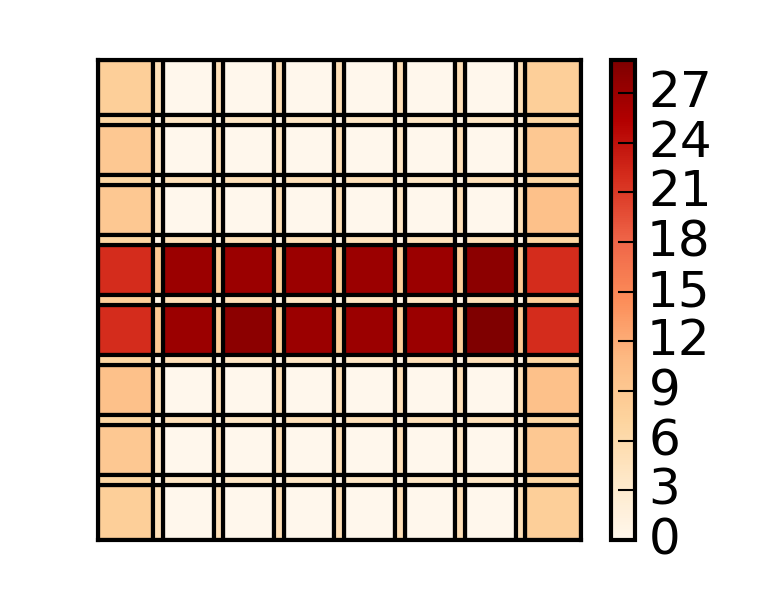}
\caption{Distribution of local basis sizes after initial training.
Relative reduction error at this configuration: $1.3 \cdot 10^{-3}$.
Script: \texttt{experiment\_draw\_basis\_sizes.py}.}
\label{fig:basis_sizes}
\end{figure}


\section{Conclusion}
We introduced ArbiLoMod, a simulation technique aiming at problems
with arbitrary local modifications and highly parallel computing
environments.
It is based on the Reduced Basis method, inheriting its advantages, but
localizing the basis construction and error estimation.
It consists of basis generation algorithms,
a localized a posteriori error estimator controlling the reduction error,
and a localized space enrichment procedure, improving the reduced
local spaces if necessary.
The initial basis generation algorithms
require no communication of unreduced quantities in a parallel implementation.
For the basis enrichment procedure, a strong compression of the
communicated quantities is possible.
We discussed the possibilities to use ArbiLoMod to implement a
scalable parallel code by reducing communication costs due to the
local structure and the local Reduced Basis strategy.
ArbiLoMod was demonstrated on a coercive example in two dimensions,
featuring high contrast, fine details and channels. Even though
small local modifications to this example lead to strong global changes in the
solution, ArbiLoMod was able to approximate the new solutions
after these geometry changes with a small fraction of the effort needed for the
initial geometry.
BSD-licensed source code is provided along with this publication, anyone
can reproduce all results presented here easily.


\section*{Acknowledgment}
{\it
The authors would like to thank Clemens Pechstein for many fruitful discussions and the referees, 
whose reviews lead to significant improvements of this manuscript.
}
\bibliographystyle{abbrv}
\bibliography{paper.bib}
\end{document}